\documentclass[reqno]{amsart}
\usepackage{amssymb}
\usepackage{amsfonts}
\usepackage{amssymb,amsmath,amsthm,tikz}

\newcommand{\Szego}{\szego}

\newcommand{\kk}{\left( \frac{k}{2\pi} \right)}

\newcommand{\h}{\hat} 
\newcommand{\la}{\langle}
\newcommand{\ra}{\rangle}
\newcommand{\bex}{\begin{example}}
\newcommand{\eex}{\end{example}}
\newcommand{\be}{\begin{equation} }
\newcommand{\ee}{\end{equation} }
\newcommand{\bcs}{\begin{cases}}
\newcommand{\ecs}{\end{cases}}
\newcommand{\Kahler}{K\"ahler }

\usepackage{kbordermatrix}

\newcommand{\brem}{\begin{rem}}
\newcommand{\erem}{\end{rem}}

\usepackage{mathtools}
\mathtoolsset{showonlyrefs}

\newcommand{\edit}[1]{{\color{red}{$\clubsuit$#1$\clubsuit$}}}

\newcommand{\bma}{\begin{pmatrix}}
\newcommand{\ema}{\end{pmatrix}}

\usepackage{color}

\usepackage{hyperref}
\def\name{Z-Z}
\hypersetup{
  colorlinks = true,
  urlcolor = black,
  pdfauthor = {\name},
  pdfkeywords = {mathematics},
  pdftitle = {\name: },
  pdfsubject = {Mathematics},
  pdfpagemode = UseNone
}
\newcommand{\pa}{\partial}

\newcommand{\baa}{\begin{align*}}
\newcommand{\eaa}{\end{align*}}
\newcommand{\bea}{\begin{eqnarray*} }
\newcommand{\eea}{\end{eqnarray*} }
\newcommand{\bee}{\begin{eqnarray} }
	\newcommand{\eee}{\end{eqnarray} }
\newcommand{\beq}{\begin{equation} }
\newcommand{\eeq}{\end{equation} }
\newcommand{\bpp}{\begin{prop}}
\newcommand{\epp}{\end{prop}}
\newcommand{\bt}{\begin{theorem}}
\newcommand{\et}{\end{theorem}}
\newcommand{\bpf}{\begin{proof}}
\newcommand{\epf}{\end{proof}}
\newcommand{\bl}{\begin{lem}}
\newcommand{\el}{\end{lem}}
\newcommand{\bc}{\begin{cor}}
\newcommand{\ec}{\end{cor}}
\newcommand{\bd}{\begin{defin}}
\newcommand{\ed}{\end{defin}}

\newcommand{\RA}{\Rightarrow}
\newcommand{\ot}{\otimes}

\newcommand{\hPi}{\hat{\Pi}}

\usepackage{latexsym}
\usepackage{graphicx}

\def\XXint#1#2#3{{\setbox0=\hbox{$#1{#2#3}{\int}$ }
\vcenter{\hbox{$#2#3$ }}\kern-.6\wd0}}

\newcommand{\wb}{\overline}

\newcommand{\acal}{\mathcal{A}}

\newcommand{\fcal}{\mathcal{F}}
\newcommand{\gcal}{\mathcal{G}}
\newcommand{\hcal}{\mathcal{H}}

\newcommand{\kcal}{\mathcal{K}}
\newcommand{\lcal}{\mathcal{L}}

\newcommand{\pcal}{\mathcal{P}}

\newcommand{\scal}{\mathcal{S}}
\newcommand{\tcal}{\mathcal{T}}
\newcommand{\wcal}{\mathcal{W}}

\newcommand{\jcal}{\mathcal{J}}

\newcommand{\Hb}{{\mathbb H}}

\def    \half   {{\frac{1}{2}}}

\def    \Z  {{\mathbb Z}}
\def    \R  {{\mathbb R}}
\def    \C  {{\mathbb C}}

 \def   \half   {{\frac{1}{2}}}

 \def    \Im     {{\operatorname{Im}}}
 \def    \Re     {{\operatorname{Re}}}
 
 \DeclareMathOperator{\Vol}{Vol}

\newcommand{\Sj}{Sj\"ostrand }

\newcommand{\szego}{Szeg\"o }

\newcommand{\kahler}{K\"ahler }

\newcommand{\wt}{\widetilde}

\newcommand{\dbar}{\bar\partial}
\newcommand{\ddbar}{\partial\dbar}

\renewcommand{\H}{{\mathbf H}}

\renewcommand{\phi}{\varphi}

\newtheorem{theo}{{\sc Theorem}}[section]
\newtheorem{cor}[theo]{{\sc Corollary}}
\newtheorem{defin}[theo]{{\sc Definition}}

\newtheorem{lem}[theo]{{\sc Lemma}}
\newtheorem{prop}[theo]{{\sc Proposition}}
\newenvironment{example}{\medskip\noindent{\it Example:\/} }{\medskip}
\newenvironment{rem}{\medskip\noindent{\it Remark:\/} }{\medskip}

\title{Pointwise Weyl law for Partial  Bergman kernels }

\author{Steve Zelditch and Peng Zhou}
\address{Department of Mathematics, Northwestern  University, Evanston, IL 60208, USA}

\email{zelditch@math.northwestern.edu}

\thanks{Research partially supported by NSF grant and DMS-1541126
and by the Stefan Bergman trust  .}

\begin{document}

\begin{abstract} This article is a continuation of a series by the authors on partial Bergman kernels and their asymptotic expasions. We prove a 2-term pointwise Weyl law for semi-classical spectral projections onto sums of eigenspaces of  spectral width $\hbar=k^{-1}$   of Toeplitz quantizations  $\hat{H}_k$ of Hamiltonians   on powers $L^k$ of a positive Hermitian holomorphic  line bundle $L \to M$ over a \kahler manifold. The first result is a complete asymptotic expansion for smoothed spectral projections in terms of periodic orbit data. When the orbit is `strongly hyperbolic' the leading coefficient defines a uniformly continuous measure on $\R$ and a semi-classical Tauberian theorem implies the 2-term expansion. As in previous works in the series, we use scaling asymptotics of the  Boutet-de-Monvel-Sjostrand parametrix and
Taylor expansions to reduce the proof to the Bargmann-Fock case. 

\end{abstract}

\maketitle

This article is part of a series \cite{ZZ16, ZZ17} devoted to partial Bergman kernels
on polarized   (mainly compact) \kahler manifolds $(L, h) \to (M^m, \omega, J)$, i.e. \kahler manifolds of (complex) dimension $m$
equipped with a Hermitian holomorphic line bundle  whose curvature form is $\omega_h = \omega$. Partial Bergman
kernels \begin{equation} \label{PBK} \Pi_{k, <E} : H^0(M, L^k)
\to \hcal_{k, <E} \end{equation}
are orthogonal projections onto proper subspaces $\hcal_{k, <E} \subset H^0(M, L^k)$ of the space
of  holomorphic sections of $L^k$. 
Let    $H \in C^{\infty}(M, \R)$ denote a classical Hamiltonian,  let $\xi = \xi_H$ denote  the Hamilton vector field of $H$, let $\nabla$ be the Chern connection. 
The quantization of $H$ is the Toeplitz Hamiltonian 
\begin{equation} \label{TOEP}  \hat{H}_k:= \Pi_{h^k} (\frac{i}{k} \nabla_{\xi} +   H) \Pi_{h^k}: 
H^0(M, L^k) \to H^0(M, L^k).\end{equation}
Here, 
$\Pi_{h^k}: L^2(M, L^k) \to H^0(M, L^k)$ is the orthogonal (\szego
or Bergman) projection.     Let
 $\{\mu_{k,j} \}_{j =1}^{d_k}$ denote the eigenvalues of $\hat{H}_k$ on the $d_k$-dimensional space $H^0(M, L^k)$ and denote the eigenspaces by
\begin{equation} \label{EIGSP} V_k(\mu_{k,j}): = \{s \in H^0(M, L^k) : 
\hat{H}_k s = \mu_{k, j} s\}. \end{equation}  Also, denote the
 eigenspace projections by
\begin{equation} \label{Pikj} \Pi_{k,j}: =  \Pi_{\mu_{k,j}} :H^0(M, L^k) \to V_k(\mu_{k,j}). \end{equation}
Then the partial Bergman kernels \eqref{PBK} are the projections onto the  spectral subspaces \begin{equation} \label{HEintro}  \hcal_{k, <E}: = \{\hat{H}_k < E\}: = \{s \in H^0(M, L^k): \langle \hat{H}_k s, s \rangle < E \langle s, s \rangle \}\end{equation} of \eqref{TOEP}.

In this article, we study the pointwise  semi-classical Weyl asymptotics of $\Pi_{k, <E}(z)$ \eqref{PBK} in 
the   conventional semi-classical scaling by $h = \frac{1}{k}$. The main results give asymptotics for the scaled pointwise Weyl sums,
\begin{equation}\label{DOSf} 
	\Pi^{E}_{k, f}(z)
	= \sum_{j} f(k(\mu_{k,j} - E)) \Pi_{k,j}(z,z)
\end{equation}
for various types of test functions $f$. Equivalently, we  
consider a  sequence of measures on $\R$,
\begin{equation} \label{mukzdef} 
d\mu_{k}^{z,1,E}(\lambda) 
=  \sum_{j}
\Pi_{k,j}(z) \delta_{k(\mu_{k,j - E)}}(\lambda).
\end{equation}
then 
$ 	\Pi^{E}_{k, f}(z) = \int_{\R} f(\lambda) d\mu_{k}^{z,1,E}(\lambda).$
When  $f \in \scal(\R)$ with $\h f \in C^\infty_c(\R)$, Theorem \ref{PikfTH} gives a complete asymptotic expansion. 
 When $f = {\bf 1}_{[a, b]}$ (the indicator function) one has sharp Weyl sums, and  Theorem \ref{2TERM} gives a  pointwise Weyl formula with 2 term asymptotics.
 
 The $\frac{1}{k}$ scaling originates in the Gutzwiller trace formula and  has been studied in numerous articles in diverse settings.
Two-term pointwise Weyl laws is a standard  topic in spectral asymptotics.  The pointwise asymptotics in the \kahler setting are quite analogous to Safarov's asymptotic results for spectral projections of
the Laplacian of a compact Riemannian manifold \cite{Sa,SV} and we use Safarov's notations to emphasize the similarity. 
For general \kahler manifolds, integrated Weyl laws and  dual  Gutzwiller trace expansions  were studied in \cite{Z97} using the Toeplitz calculus of \cite{BG81}.  Pointwise Weyl laws of the type studied in this article are given in  Borthwich-Paul-Uribe \cite{BPU98}, based on the Boutet-de-Monvel-Guillemin Hermite Toeplitz calculus \cite{BG81}

The main purpose of this paper is to prove pointwise Weyl asymptotics using the techniques developed in \cite{ZZ16,ZZ17}.
Existence of an asymptotic expansion for smoothed Weyl sums is a straightforward consequence of a parametrix construction and  of the method of statonary phase, replacing the elaborate symplectic spinor symbol calculus   of \cite{BG81}. However, the coefficients are complicated to compute. In the Toeplitz theory of \cite{BG81, BPU98} they
   are calculated using the symplectic spinor symbol calculus of Toeplitz operators, while we use scaling asymptotics of the quantized flow in the sense of \cite{ShZ02,RZ,ZZ17}. It is  shown that the leading coefficients  depend only on the quadratic part of the Taylor expansions. Hence, the coefficients are the same as in the linear model of \cite{Dau80} once the flow is linearized at a period.  Our approach gives a somewhat simpler formula for the leading term than in \cite{BPU98} and it is not completely obvious that the formulae agree; in Section \ref{BPUSECT} we show that the formulae do agree with those of \cite{BPU98}.
  Related calculations using the scaling  approach of  this article  are also given  in articles of Paoletti  \cite{P12,P14}.

 In the previous articles, we studied the scaling asymptotics of   $\Pi_{k,<E}(z): = \Pi_{k,<E}(z,z) $  in a $\frac{1}{\sqrt{k}}$-tube around  the interface $\partial \acal$ between the allowed and forbidden regions,
\begin{equation} \label{acalDef} \acal: = \{z: H(z) < E\},
\quad \fcal = \{z: H(z) >E\}.
\end{equation} 
This $\frac{1}{\sqrt{k}}$ scaling was the new feature of the Weyl asymptotics of \cite{ZZ17} and is reminiscent of the scaling of the central limit theorem.  The  $\frac{1}{k}$-scaling was also studied in  \cite{ZZ17}, but it  was sufficient for the purposes of that article to obtain the  crude asymptotics corresponding to the singularitiy of the Fourier transform
$\widehat{d \mu_k^{z, 1, E}}(t)$ at $t = 0$.  Technically speaking, the main difference with respect to \cite{ZZ17} is that the asymptotics of the $\frac{1}{\sqrt{k}}$ scaling only involve `Heisenberg translations' while those of $d\mu_{k}^{z,1,E}$ involve the metaplectic representation.  Although the notation and approach of this article have considerable overlap with \cite{ZZ17} we give a rather detailed exposition for the sake of completeness.

   \subsection{Statement of results}

To state the results, we need some further notation. Given a Hermitian metric $h$ on $L$, we denote by $X_h = \partial D_h^* \subset L^*$ the unit $S^1$ bundle $\pi: X_h \to M$ over $M$ defined as the boundary of the unit co-disc bundle in the
dual line bundle $L^* $ to $L$. As reviewed in Section \ref{CRSECT}, $X_h$ is a strictly pseudo-convex CR-manifold, and we denote the CR sub-bundle by $HX \subset TX_h$. As reviewed in Section \ref{LIFT},
the  Hamilton flow  $g^t: M \to M$  lifts to   a contact flow $\hat{g}^t: X_h \to X_h$  (Lemma \ref{gtform}) with respect to the contact structure $\alpha$ associated to the \kahler potential of $\omega$.  Then $HX = \ker \alpha$ and therefore $D \hat{g}^t: HX \to HX$.
Moreover, $HX$ inherits a complex structure $J$ from that of $M$
under the identification $\pi_*:  H_{x}X \to T_{\pi(x)} M$, for all $x \in X$. Its complexification has a splitting $H_xX_\C = H_xX\otimes \C = H_{x}^{1,0}X \oplus H_{x}^{0,1}X$ into subspaces of types $(1,0)$ resp. $(0,1)$. In the generic case where  $\hat{g}^t$ is non-holomorphic, it does not preserve this splitting.

At each  point $x \in X$, the complexified CR subspace  $HX_{\C}$ equipped with $J_{x}$ together with the Hermitian metric $h_{x}$ determines an {\it osculating  Bargmann-Fock space} $\hcal_{J_x}$ (see  Section \ref{CRSECT} and Section \ref{OBFSECT}  for background).      Thus,  $\hcal_{J_{x}}$ 
 is the space  of entire holomorphic functions on 
$H_{x}^{1,0}X $ which are square integrable with respect to the {\it ground state} $\Omega_{J_{x}}$ (defined in  \eqref{GSJ}).  Symplectic transformations $T: H_{x} X \to H_{x} X$ resp. $T: T_z M \to T_z M$  may be quantized by the metaplectic representation 
as complex linear symplectic maps (see  \eqref{eta} and Section \ref{METASECT}) on the osculating Bargmann-Fock space, 
\begin{equation} \label{WJ}  W_{J_x}(T): \hcal_{J_x}  \to \hcal_{J_x}. \end{equation}

The asymptotics of $\mu^{z, 1, E}_{k}(f)$  depend on whether or not $z \in M$ is a periodic point for $g^t$.
\begin{defin} Define periodic points of $g^t$,
as follows:  $$\pcal_E := \{z \in H^{-1}(E): \exists T > 0: g^T z
= z\}.$$ 
For $z \in \pcal_E$, let $T_z$ denote the minimal period
$T > 0$ of $z$.
\end{defin} 

It may occur that $z \in \pcal_{E}$ but the orbit $g_h^t(x)$ with $\pi(x) = z$ is not periodic, where $g_h^t$ is the flow generated by the horizontal lift $\xi_H^h$ of the Hamiltonian vector field $\xi_H$. This is due to holonomy effects: parallel translation of sections of $L^k$ around the closed curve $t \mapsto g^t(z)$ may have non-trivial holonomy. We denote the holonomy by \begin{equation} 
\label{HOLONOMYDEF} e^{i n\theta_z^h}: = {\rm the\; unique\; element\;} e^{i \theta} \in S^1: g_h^{n T_z}  x =r_{\theta} x. \end{equation} 

  Let $z \in \pcal_E$, $T= n T_z$ be a period for $n \in \Z$.  
  Then $D g^{T} _{z}$ induces
 linear symplectic map 
\begin{equation} \label{DGnMa} S:= D g^{T} _{z}: T_z M \to T_z M, \end{equation} 
When working in the \kahler context it is better to conjugate to the complexifications,
$$
 T_z M \otimes \C = T^{1,0}_z M \oplus 
T^{0,1}_z M. $$
 We denote the projection to the `holomorphic component' by
\begin{equation} \label{pi10} \pi^{1,0} : T_z M \otimes \C \to T^{1,0} M. \end{equation}
The spaces $T^{1,0}M, T^{0,1}M$ are paired complex Lagrangian subspaces.

  Relative to a  symplectic basis $\{e_j, J e_k\}$   of $T_z M$ in which $J$ assumes the standard form $J_0$,   the matrix of $D g^{n T_z}$ has the form,
\be  \label{ABCD} D g^{n T_z}_{z} := S^n: =  \begin{pmatrix} A_n & B_n\\ & \\ C_n & D_n \end{pmatrix} \in Sp(m, \R). \ee
If we conjugate to the complexifcation 
$T_z M \otimes \C$ by the natural map $\wcal$  defined in \eqref{PQDEF}, then \eqref{ABCD} conjugates to $$ \bma P_n & Q_n \\ \bar Q_n & \bar P_n \ema \in Sp_c(m). $$ The holomorphic block
\begin{equation} \label{PDEF} P_n =  \left(A_n + D_n+ i (-B_n + C_n) \right) =  \pi^{1,0} \;\wcal S^n \wcal^{-1} \;\pi^{1,0}: 
T^{1,0}_z M \to T^{1,0}_z M\end{equation}
plays a particularly important role.

The symplectic map \eqref{DGnMa} is quantized by the metaplectic representation $W_{J_z}$ \eqref{WJ}  (see Section \ref{METASECT}) on the  osculating Bargmann-Fock space $\hcal_{J_z}$ of square integrable holomorphic functions on $T^{1,0}_z M$, that is,
 the metaplectic representation  defines a unitary operator \begin{equation} \label{WJgM}  W_{J_{z}} \;(D g^{n T_z}_{z}): \hcal_{J_z}(T_z^{1,0} M) \to \hcal_{J_z}(T_z^{1,0} M). \end{equation}

The  two-term Weyl law is stated in terms of certain   data associated
to $D g^{nT_z}$  and $W_{J_z}(D g^{nT_z})$ \eqref{WJgM}.  First, we let $\wcal \xi_H$ be the image of the Hamilton vector field  $\xi_H$ in $T_z M \otimes \C$. Let $\alpha = \pi^{1,0} \wcal \xi_H$, let $\bar{\alpha} \in \pi^{0,1} \wcal \xi_H$,
and let $P_n$ be as in \eqref{PDEF}. Set,
\begin{equation} \label{gcalndef} \gcal_n(z): = (\det P_n)^{-\half} \cdot (\bar{\alpha} \cdot  P_n^{-1} \alpha) ^{-\half}. \end{equation}
The factor $(\det P)^{-\half}$ has an interpretation, 
\begin{equation} \label{ABCDintro} (\det P_n)^{-1/2} =   \langle W_{J_{x}} \;(D g^{n T(x)}_{x}) \;\Omega_{J_z}, \Omega_{J_z} \rangle \end{equation}
as the matrix element 
of  \eqref{WJgM}
relative to the ground state $\Omega_{J_{z}}$ in $\hcal_{J_{z}}$. This relation  is essentially proved by Bargmann and by Daubechies \cite{Dau80}. 
It can be proved by  comparing the Bargmann-Fock  metaplectic representation of Section \ref{METASECT} with Daubechies' Toeplitz construction of metaplectic representation 
in Section \ref{TOEPMETA}.  Daubechies did not explicitly use the conjugation $\wcal$ to the complexification, and therefore did not record the identity \eqref{ABCDintro}.

Also let $e^{i n\theta_x^h}$ denote the holonomy of the horizontal lift of the orbit $t \to g^{t}(z)$  at $t = n T_z$. 
We  define the function  $Q^E_{z, k}(s)$ by:

\begin{defin} \label{QDEF}  \begin{equation} \label{Q} 
 Q^E_{z, k}(s) = \left\{\begin{array}{ll}\gcal_0(z) 
 & 
  z \; \notin \pcal_E \\ & \\  \sum_{n \in \Z} (2\pi)^{-1} e^{-i n T_z s}e^{-in k \theta_z^h} 
\;\gcal_n(z)  &

 \; z \in \pcal_{E}.\end{array}
\right. \end{equation}

\end{defin}

\begin{defin} For $z \in \pcal_E$, define the distributions  $d\nu_k^z$ on $f \in \scal(\R)$ by
 $$\int_{\R} f(\lambda) d\nu_k^z(\lambda) =
\sum_{n \in \Z}   \hat{f} (n T_z)
\;\gcal_n(z) e^{-i nk \theta_{z}^h} = \int_{\R} f(s)  Q^E_{z,k}(s) ds$$
\end{defin}

The nature of $Q_{z, k}(s) $ and $\nu_k^z$ depends on the type of periodic orbit
of $z \in \pcal_E$. 
 In this article we confine ourselves to the case where the orbit of $z$ is
 `real positive definite symmetric' in the following sense:
\begin{defin} \label{HYPERBOLIC} Let $z \in \pcal_E$, with $T_z = T$,   and let $(T_z M, J_z, \omega_z)$ be the tangent space equipped with its complex structure and symplectic structure. Let $\{e_j, f_k\}_{j, k =1}^m$ be a symplectic basis of $T_z M$ in which
$J = J_0$ and $\omega = \omega_0$ take the standard forms. We
say that $D G^T_z$ is {\it positive definite symmetric symplectic} if its matrix $S \in Sp(m, \R)$ in 
the basis $\{e_j, f_k\}_{j, k = 1}^m$ is a symmetric positive definite symplectic matrix. 

\end{defin} 
Positive definite symplectic matrices are discussed in Sections \ref{SLA}- \ref{PDSS} and in Section \ref{MEPDSS}. They are diagonalizable by orthogonal matrices in $O(2n)$
and  by unitary matrices in $U(n)$. In invariant terms, $O(2n)$ is the
orthogonal group of $(T_z M, g_{J_z})$ where $g_{J_z}(X, Y) = \omega_z(X, J_z Y)$. Unitary matrices commute with $J_z$. The eigenvalues of $DG_z^T$ are real and to come in inverse pairs. The eigenvalue $1$ corresponds to the Hamilton vector field $\xi_H$ of $H$ and there is a second eigenvector of eigenvalue $1$ coming from the fact that  periodic orbits come in 1-parameter families  (symplectic cylinders) as the energy level $E$ is varied (see \cite{AM78}). The  eigenvalues in the symplectic orthogonal complement of the eigenspace $V(1)$ of eigenvalue $1$ come
in unequal real inverse pairs  $\lambda, \lambda^{-1} $.
 For expository simplicity, we omit the case where eigenvalues are complex of modulus $\not= 1$ and arise in $4$-tuples 
$\lambda, \lambda^{-1}, \bar{\lambda}, \bar{\lambda}^{-1}$ (sometimes
called loxodromic). We do discuss the elliptic  case where $S \in U(n)$, and thus all of  the eigenvalues have modulus $1$ and come in complex conjugate pairs.

We refer to \cite{deG} for background on positive definite
symmetric symplectic matrices and to \cite{L02} for  types of periodic orbits of  Hamilltonian flows.

\bd\label{hyperhypo}
We say that $z$ satisfies the  {\em strong hyperbolicity hypothesis} 
if  $D g_z^T: (T_z M,J_z) \to (T_z M,J_z)$ is a positive symplectic map, with a 2-dimensional symplectic eigenspace $V(1)$  for the eigenvalue $1$. 
\ed

The main motivation for this hypothesis is that we can explicitly compute
\eqref{gcalndef} in this case (see Proposition \ref{PROPME}). Almost the same computation works if $Dg^T_z$ is unitary (the elliptic case) However,
in the strong hyperbolic case, we can prove that the infinite series
defining \eqref{Q} converges absolutely and uniformly, and therefore:

\begin{prop} \label{HYPLEM} If $z$ satisfies the strong hyperbolicity
hypothesis, 
then $\nu_k^z$ is an absolutely  continuous measure. \end{prop}
 
The main result is a   sharp 2-term Weyl law in this case: 

\begin{theo} \label{2TERM} Assume that $z \in H^{-1}(E)$ and that $z$
satisfies the strong hyperbolicity hypothesis. Then,
$$ \int_a ^b d\mu_{k}^{z,1,E}= \left\{ \begin{array}{ll}  \kk^{m-1/2} \gcal_0(z) (b-a)  (1+ o(1)))., &z \in H^{-1}(E), \; z \notin \pcal_E\\ &\\   \kk^{m-1/2}  \nu_k^z(a,b)(1+o(1)), & z \in H^{-1}(E),  z\in \pcal_E, 
\end{array} \right..$$

\end{theo}
 Theorem \ref{2TERM}  is a \kahler Toeplitz analogue of \cite[Theorem 1.8.14]{SV} (originally proved in \cite{Sa}).
The difference between $z \notin \pcal_E$ and  $z \in \pcal_E$ is that in the former case,   there is a contribution only from the $t = 0$ times of $g^t$ (the identity map) and in the latter case there are contributions from all iterates
of $g^{T_z}$.

It may be expected that Theorem \ref{2TERM} extends in some suitable way
to any type of periodic orbit. In the somewhat analogous Riemannian setting
studied in \cite{SV}, the pointwise Weyl
law involves first return maps on the set of  geodesic loop directions $\xi \in S^*_x M$ at a point $x \in M$ rather than closed orbits.  In some cases (such as where $x$ is a focus of an ellipsoid), the corresponding measures or $Q$-functions are calculated in \cite[Example 1.8.20]{SV}. Otherwise, the authors say simply that it is difficult to determine when the ``$Q$'' function of \cite[(1.8.11)]{SV} is uniformly continuous. It is likely that Theorem \ref{2TERM} can be extended to any orbit for which none of the eigenvalues
on the symplectic orthogonal complement of the $V(1)$-eigenspace of
$S$ have modulus one. This is certainly the case, by the same proof as in Proposition \ref{HYPLEM}, if $S$ is diagonalizable by a unitary matrix.

\subsection{Outline of the proof}

The proof is a continuation of that in  \cite{ZZ17}, adding information on the remainder term and its relation to periodic orbits of periods $T > 0$.
 Given a function $f \in \scal(\R)$ 
(Schwartz space) one defines \begin{equation} \label{fofHFT}
f(k \hat{H}_k) = \int_{\R} \hat{f}(\tau) e^{i k \tau
\hat{H}_k } d\tau =  \int_{\R} \hat{f}(t) U_k(t) dt,
\end{equation} 
where  \begin{equation} \label{UkDEF} U_k(t) = \exp i t k \hat{H}_k. \end{equation} is the unitary group on $H^0(M, L^k)$ generated by $k \hat{H}_k$. Note that  $f(k \hat{H}_k)$ is the operator  on $H^0(M, L^k)$ with the same eigensections
as $\hat{H}_k$ and with eigenvalues $f(k\mu_{k,j})$. The metric contraction of the Schwarz kernel on the diagonal is given by,
\begin{equation} \label{fofHFTintro} \Pi_{k,f}^E(z) = 
\int_{\R} \hat{f}(t)  e^{- i k t E} e^{i k t
\hat{H}_k }(z,z) dt =  \int_{\R} \hat{f}(t) e^{- i k t E}  U_k(t,z) dt.
\end{equation} 
Here, and henceforth, the metric contraction of a kernel
$K_k(z,w)$ is denoted by $K(z)$. \begin{defin}\label{CONTDEF} The metric contraction of a kernel  $M_k(z,w) : =  \sum_{j=1}^{d_k} \mu_{k,j} s_{k,j}(z) \overline{s_{k, j}(w)}$
expressed in an orthonormal basis $\{s_{k,j}\}_{j=1}^{d_k} $ of $H^0(M, L^k)$ 
is defined by 
$$M_k(z): = \sum_{j=1}^{d_k} \mu_{k,j} |s_{k,j}(z)|_{h^k}^2, \;\; (d_k = \dim H^0(M, L^k)) $$
\end{defin}
In Section  \ref{LIFT} below, we lift sections and kernels to the associated 
$U(1)$ frame bundle of $L^*$; then metric contractions are the same as values of the lifts along the diagonal.

In \cite{ZZ17} it is shown that $U_k(t)$ is a semi-classical Toeplitz Fourier integral operator of a type defined
in \cite{Z97}. As in \cite{ZZ17} we construct a parametrix 
 of the form,
\begin{equation} \label{UktPARA} \hat{\Pi}_{h^k} \sigma_{k,t} (\hat{g}^{-t})^* \hat{\Pi}_{h^k} \end{equation} where $(\hat{g}^{-t})^*$ is the pullback of functions on $X_h$
by $\hat{g}^t$ and where $\sigma_{k,t}$ is a semi-classical symbol originally calculated in \cite[Unitarization Lemma 1 (2b.5) and (3.10)]{Z97}. In fact, to leading order in $k$,  \begin{equation}\label{sigmakt}
\sigma_{k t}(z)  = \langle \Omega_{Dg^T_z J_z}, \Omega_{J_{g^t z}} \rangle^{-\half}. \end{equation}
Here, $D g^T J_z$ is the image of the complex structure at $z$ and 
$J_{g^t z}$ is the complex structure of $T_{g^t z} M$ and $\Omega_J$ denotes the ground state in the Bargmann-Fock Hilbert space with complex structure $J$. It was proved in \cite{Dau80,Z97} that \eqref{sigmakt} equals
$(\det P)^{-\half}$  by calculating the inner product of the two Gaussians.

Combining \eqref{mukzdef} and \eqref{fofHFTintro} shows that
\begin{equation} \label{RHOSQRTintrob} \mu^{z,1,E}_{k}(f): = \int_{\R} f(x) d\mu_{k}^{z,1,E} =   \int_{\R} \hat{f}( t) e^{- i E k  t} \hat{\Pi}_{h^k} \sigma_{k t} (\hat{g}^t)^* \hat{\Pi}_{h^k}(z) dt, \end{equation}  
or equivalently 
\begin{equation} \label{RHOSQRTFT} \widehat{\mu^{z,1,E}_{k}}(t)  = e^{- i E k  t} U_k(t,z,z). \end{equation}  
Using a semi-classical   Tauberian theorem, it is proved in Section \ref{TAUBERPROOF} that the singularities of \eqref{RHOSQRTFT} determine the 2-term asymptotics of $\mu^{z,1,E}_{k}[a,b]$ for any interval. 
Proposition \ref{HYPLEM} follows because  the singularities are of a different type depending on the convergence of $Q_z(k)$.



To prove the two-term Weyl law, we begin by obtaining asymptotics for  the  smoothed partial density of states \eqref{RHOSQRTintrob}. In the first case
where $z \notin \pcal_E$, the only singularity occurs at $t = 0$ and so the
expansion is the same as in  \cite[Theorem 3]{ZZ17}(recalled here as Theorem \ref{ELLSMOOTH}).  The time interval $[-\epsilon, \epsilon]$ is assumed to be so short that it contains no non-zero periods of  periodic orbits When $z \notin
H^{-1}(E)$ the expansion is rapidly decaying. Thus, the new aspect is the
second case where $z \in \pcal_E$.

\begin{theo}
 For $ f \in \scal(\R)$ with $\hat f \in C^\infty_c(\R)$, we have (see Definitions \ref{HYPERBOLIC} and \ref{CONTDEF})
\label{PikfTH} 
$$ \Pi_{k, f}(z): = \int_{\R} f d\mu_{k}^{z,1,E} = \left\{ \begin{array}{ll}  \kk^{m-1/2} \h f(0) \gcal_0(z)(1+ O(k^{-1}))., &z \in H^{-1}(E), \; z \notin \pcal_E\\ &\\ \kk^{m-1/2}  \sum_{n \in \Z}  \hat{f}(n T_z)
\;\gcal_n(z) e^{-i k n\theta_z^h} + O(k^{m-3/2}), & z \in H^{-1}(E),  z\in \pcal_E, \\ &\\
O(k^{-\infty}), & z \notin H^{-1}(E) \end{array} \right. $$
\end{theo}

To prove Theorem \ref{PikfTH} we use  the Boutet de Monvel-\Sj  parametrix for $\hat{\Pi}_{h^k}.$  This gives a parametrix for  \eqref{fofHFT} and \eqref{RHOSQRTintrob}  as   semi-classical oscillatory
integrals with complex phases. The  phase has no critical points when the orbit does not lie in $H^{-1}(E)$ and no critical points for $t \not= 0$ when
$z \notin \pcal_E$.  The main difficulty is to evaulate or interpret the phases and the  Hessian determinant (and other invariants that arise) dynamically, and to determine whether or not they are invariants of 
 $D \hat{g}^T$ or invariants of the full orbit.  One phase factor is   a holonomy integral around the periodic orbit $\hat{g}^t(x)$.
    In Proposition \ref{phase-quad} it is shown that althouth the holonomy  is apriori a  `global invariant' of the orbit rather than an invariant
 of the first return map, in fact    the Hessian of the  holonomy can be expressed as an invariant of the first return map.

    To evaluate the Hessian determinants, we first do so in the linear Bargmann-Fock setting, where 
 $H$ is a quadratic Hamiltonian on the \kahler manifold  $\C^m$,  equipped with a general complex structure $J$ and a Hermitian metric $h$.

\begin{prop}\label{WJzz}  Let $H$ be a quadratic Hamiltonian in  the Bargmann-Fock setting. Assume that $H$ has compact level sets
and non-degenerate periodic orbits on level $E$. Then, in the notation of
Definition \ref{CONTDEF}, $$\int_{\R} \hat{f}(t) U_k(t, z) e^{- it E k} dt \simeq   \kk^{m-\half}   \sum_{n \in \Z} \hat{f}(n T_z) e^{-i k \theta_{z n T_z}}  (\bar \alpha P_n^{-1} \alpha)^{-1/2}  (\det P_n)^{-1/2}, $$
where $P_n$ is the holomorphic block of $D g^{n T_z}$ \eqref{PDEF} and $\pi^{1,0} \wcal \xi_H = \alpha$.
\end{prop}

We give a detailed proof in Section \ref{BFREDUCTION}  because the general case is reduced to the Bargmann-Fock case. 
 It is shown in this article that the linearized calculation is the principal symbol of non-linear problem \eqref{RHOSQRTintrob}, hence that Theorem \ref{2TERM} can be reducedd to Proposition \ref{WJzz}. The proof consists of nothing more than Taylor expansions of the phase in suitable \kahler normal coordinates and stationary phase.

\section{Background}

The background to this article is largely the same as in \cite{ZZ17}, and we refer there for many details.  Here we give a quick review to setup the notation. First we introduce co-circle bundle $X \subset L^*$ for a positive Hermitian line bundle $(L,h)$, so that holomorphic sections of $L^k$ for different $k$ can all be represented in the same space of CR-holomorphic functions on $X$, $\hcal(X) = \oplus_k \hcal_k(X)$. The Hamiltonian flow $g^t$ generated by $\xi_H$ on $(M,\omega)$ will be lifted to a contact flow $\h g^t$ generated by $\h \xi_H$ on $X$. Then we review the Toeplitz quantization for a contact flow on $X$ following \cite{Z97, RZ}. 

\subsection{\label{SLA} Symplectic Linear Algebra}
Let $(V, \sigma)$ be a real symplectic vector space of dimension $2n$ and
let $J $ be a compatible complex structure on $V$. There exists a symplectic basis in which  $V \simeq \R^{2m} $,  $\sigma$ takes the standard form $ \omega = 2 \sum_{j=1}^m dx_j \wedge d y_j$ and   $J$ has the standard form,  $J_0 = \begin{pmatrix}  0 & - I \\ & \\
I & 0 \end{pmatrix}.$
Let $H^{1,0}_J$ resp. $H^{0,1}_J$, denote the $\pm i$ eigenspaces of $J$ in $V \otimes \C$. The projections onto these supspaces are denoted
by
$$P_J = \half(I - i J): V \otimes \C \to  H^{1,0}_J, \;\; \bar{P}_J = \half(I + i J): V \otimes \C \to H^{0,1}_J. $$

Let $S \in Sp(m, \R)$ be a real symplectic matrix. Then its  transpose $S^t = J S^{-1} J^{-1}$ also lies in  $Sp(m, \R)$ and   $S J = J (S^t)^{-1}.$
\subsection{\label{PDSS} Symmetric symplectic matrices}

A matrix $S$ is called
a symmetric symplectic matrix if $S \in Sp(n, \R)$ and $S^t = S$. 
For such $S$ it follows that $SJ = J S^{-1}$.
A good refernce for positive definite symplectic matrices is \cite[p. 6]{L02} and \cite[p. 52]{L02}.
For the following see \cite[Proposition 22]{deG}. Let 
$U(n) = Sp(n) \cap O(2n, \R)$. Then $UJ = JU$ and 
$$U = \begin{pmatrix} A & -B \\ & \\
B & A \end{pmatrix},\;\; A B^t = B^t A, \;\; A A^t + BB^t = I, \;\; U^{-1} = \begin{pmatrix} A^t &B^t \\ & \\
-B^t & A^t \end{pmatrix} = U^t. $$

\begin{prop} If $S$ is a positive definite symmetric symplectic matrix and $\Lambda = \rm{diag} (\lambda_1, \dots, \lambda_n; \lambda_1^{-1}, \dots, \lambda_n^{-1})$ is the given diagonal
matrix, then there exists $U \in U(n)$ so that $S = U^t \Lambda U. $
\end{prop}



The following is \cite[Proposition 26]{deG}

\begin{prop} A symplectic matrix $S$ is symmetric positive definite if and only if $S = e^X$ with $X \in \mathfrak sp(n)$
and $X = X^t$.  The map $\exp: \mathfrak sp(n) \cap Sym(2n, \R) \to Sp(n) \cap Sym_+(2n, \R)$ is a diffeomorphism.
\end{prop}

If $e_1, \dots, e_n$ are orthonormal eigenvectors of $S$ corresponding to the eigenvalues $\lambda_1, \dots, \lambda_n$
then since $S J = JS^{-1}$, 
$$S J e_k = J S^{-1} e_k = \frac{1}{\lambda_j} J e_k. $$
Hence $\pm J e_1, \dots, \pm J e_n$ are orthonormal eigenvectors of $U$ corresponding to eigenvalues
$\lambda_1^{-1}, \dots, \lambda_n^{-1}$ and 
$\begin{pmatrix} A \\ B \end{pmatrix} = [e_1, \dots, e_n]. $



\subsection{\label{BFVS} The Bargmann-Fock space of a complex Hermitian vector
space}

The Bargmann-Fock spaces can be defined more generally for any complex structure $J$ on $\R^{2n}$  and any Hermitian metric on $\C^n$.

Let $(V, \omega)$ be a real symplectic vector space. Define 
$$\jcal = \{J : \R^{2n} \to \R^{2n}, \;\; J^2 = -I, \;\; \omega(JX, JY) = \omega(X, Y),
\;\; \omega(X, J  X)>> 0\}$$
to be the space of complex structures on $\R^n$ compatible with $\omega$.
 The Bargmann-Fock space of a symplectic
vector space $(V, \sigma)$ with compatible complex structure $J \in \jcal$ is the Hilbert space,
$$\hcal_{J} = \{f  e^{-\half \sigma(v, J v)} \in L^2(V, dL), f\; \mbox{is\; entire \; J-\;holomorphic}\}. $$ 
Here,    \begin{equation} \label{GSJ} \Omega_{J} (v) := e^{-\half \sigma(v, J v)}
\end{equation} is  the `vacuum state' and $d L$ is normalized Lebesgue measure (normalized
so that square of the symplectic Fourier transform is the identity).
The orthogonal projection onto $\hcal_J$ is denoted by $P_J$ in \cite{Dau80} but we
denote it by $\Pi_J$ in this article. Its Schwartz kernel relative to $dL(w)$ is denoted
by $\Pi_J(z,w)$.

\begin{rem} The Bargmann-Fock space with $J = i$ the standard complex structure is often defined instead as the weighted Hilbert space
of entire holomorphic functions with Gaussian weight $C_n e^{- |z|^2} dL(z)$ where
$C_n$ is a dimensional constant. In  this definition the vacuum state is $1$. There
is a natural isometric `ground states' isomorphism to $\hcal_J$ defined by multiplying by $\sqrt{\Omega_J}$. With the Gaussian measure, the Bergman kernel is $B(z,w) = e^{z \cdot \bar{w}}$. When $V = \C^n$  we write $v = Z$, $J Z = i Z$,  and $\sigma(Z,W) = \Im \overline{Z} \cdot W$. 
Then $\Omega_J(Z) = e^{- \half |Z|^2}.$
\end{rem}

\subsection{Bargmann-Fock Bergman kernels}

For BF model, we have
$\Pi_k: L^2(M, L^k) \to H^0(M, L^k)$ the Bergman projection operator. And $\h \Pi_k: L^2(X) \to \hcal_k(X)$, the Szego projection operator on $X$ to Hardy space's Fourier component. Let $H$ also denote its pull back on $X$.

The semi-classical Bargmann-Fock Bergman kernels \eqref{Pikdef} on $\C^n$ are given  by
\begin{equation} \label{SCBFPI}  \Pi^{\C^m}_{k, h_0, i}(z,w) = \kk^{m} e^{k (z \bar w - |z|^2/2 - |w|^2/2)}. \end{equation}
Their lifts to $X$ are given by
\[ \hat{\Pi}^{ \C^m}_{k, h_0, i} (\hat{z}, \hat{w}) = \kk^{m} e^{k \psi(\h z, \h w)} \] where 
\begin{equation} \label{BFphase}  \h \psi(\h z, \h w) = i (\theta_z - \theta_w)  + \psi(z,w) = i(\theta_z - \theta_w) +  z \cdot \bar w - |z|^2/2 - |w|^2/2. \end{equation}
where  $\h z = (\theta_z, z) \in S^1 \times M \cong X$ denotes a lift of $z$. \footnote{We also use the notation $x = (z, \theta_z), y = (w, \theta_w)$}

In the general case, by  (3.1) of \cite{Dau80}, one has
\begin{equation} \Pi_J \psi(z) = \langle \Omega_J^z, \psi \rangle = \int_{\C^n} \psi(v)
\overline{\Omega_J^z}(v) dv,\end{equation} i.e.
\begin{equation}\label{PIBF} \Pi_J(z,w) = \overline{\Omega_J^z}(w) = e^{ i \sigma(z,w)} e^{- \half \sigma(z - w, J (z-w))}   \end{equation}
which redcues to $  e^{i \Im z \bar{w}} e^{- \half (|z - w|^2)} =  e^{z \bar{w}} e^{- \half (|z |^2 + |w|^2)}$ in the case $J = i, h =h_0$.

\subsection{\label{CRSECT}Holomorphic sections in $L^k$ and CR-holomorphic functions on $X$}
Let $(L,h) \to (M, \omega)$ be a positive Hermitian line bundle, $L^*$ the dual line bundle. Let 
\[ X := \{ p \in L^* \mid \|p\|_h = 1\}, \quad \pi: X \to M \]
be the unit circle bundle over $M$. 

Let $e_L \in \Gamma(U, L)$ be a non-vanishing holomorphic section of $L$ over $U$, $\varphi = -\log \|e_L\|^2$ and $\omega = i \ddbar\varphi$. We also have the following trivialization of $X$:
\be \label{X-triv} U \times S^1 \cong X|_U, (z; \theta) \mapsto e^{i\theta} \frac{e_L^*|_z}{\|e_L^*|_z\|}. \ee

$X$ has a structure of a contact manifold. Let $\rho$ be a smooth function in a neighborhood of $X$ in $L^*$, such that $\rho>0$ in the open unit disk bundle, $\rho|_X=0$ and $d\rho|_X\neq 0$. Then we have a contact one-form on $X$
\begin{equation} \label{alphadef} \alpha =- \Re(i\dbar\rho)|_X, \end{equation}
well defined up to multiplication by a positive smooth function. We fix a choice of $\rho$ by
\[ \rho(x) = - \log \|x\|_h^2, \quad x \in L^*, \]
then in local trivialization of $X$ \eqref{X-triv}, we have
\be \alpha =   d \theta   - \frac{1}{2}  d^c \varphi(z). \label{alpha-def} \ee

$X$ is also a strictly pseudoconvex CR manifold. The {\it  CR structure} on $X$ is defined as
follows:
 The kernel of $\alpha$ defines a horizontal hyperplane bundle \begin{equation} \label{HDEF} HX :=
\ker \alpha \subset TX, \end{equation}  
invariant under $J$ since $\ker \alpha = \ker d \rho \cap \ker d^c\rho$. Thus we have a splitting
\[ TX \ot \C \cong H^{1,0} X \oplus H^{0,1}X \oplus \C R.\]
A function $f: X \to \C$ is CR-holomorphic, if $df|_{H^{0,1}X} = 0$.

A holomorphic section $s_k$ of $L^k$ determines a CR-function $\h s_k$ on $X$ by
\[ \h s_k(x) := \la x^{\ot k}, s_k\ra, \quad x \in X \subset L^*. \]
Furthermore $\h s_k$ is of degree $k$ under the canonical $S^1$ action $r_\theta$ on $X$, $\h s_k(r_\theta x) = e^{i k \theta} \h s_k(x)$. The inner product on $L^2(M,L^k)$ is given by
\[ \la s_1, s_2 \ra := \int_M h^k(s_1(z), s_2(z)) d \Vol_M(z), \quad d \Vol_M = \frac{\omega^m}{m!}, \]
and inner product on $L^2(X)$ is given by
\[ \la f_1, f_2 \ra := \int_X f_1(x) \wb{f_2(x)} d \Vol_X(x), \quad d \Vol_X =  \frac{\alpha}{2\pi}\wedge\frac{(d\alpha)^m}{m!}. \]
Thus, sending $s_k \mapsto \h s_k$ is an isometry.

\subsection{\Szego kernel on $X$}
On the circle bundle $X$ over $M$, we define the orthogonal projection from $L^2(X)$ to the CR-holomorphic subspace $\hcal (X) = \h \oplus_{k \geq 0} \hcal_k(X)$, and degree-$k$ subspace $\hcal_k(X)$: 
\[ \h \Pi: L^2(X) \to \hcal(X), \quad \h \Pi_k: L^2(X) \to \hcal_k(X), \quad \hPi = \sum_{k \geq 0} \hPi_k. \]
The Schwarz kernels $\hPi_k(x,y)$ of $\hPi_k$ is called the degree-$k$ \Szego kernel, i.e. 
\[ (\hPi_k F)(x) = \int_X \hPi_k(x,y) F(y) d \Vol_X(y), \quad \forall F \in L^2(X). \]
If we have an orthonormal basis $\{\h s_{k,j}\}_j$ of $\hcal_k(X)$, then
\[ \hPi_k(x,y) = \sum_j \h s_{k,j}(x) \wb{ \h s_{k,j}(y)}. \] 

The degree-$k$ kernel can be extracted as the Fourier coefficient of $\hPi(x,y)$
\begin{equation} \label{Pikdef}  \hPi_k(x,y) = \frac{1}{2\pi} \int_0^{2\pi} \hPi(r_\theta x, y) e^{-i k \theta} d \theta. \end{equation}
We refer to \eqref{Pikdef} as the {\it semi-classical Bergman kernels}.

\subsection{Boutet de Monvel-Sj\"ostrand parametrix for the \Szego kernel}

Near the diagonal in $X \times X$, there exists a parametrix due to  Boutet de Monvel-Sj\"ostrand 
\cite{BSj} for the \Szego kernel of the form,  
\begin{equation} \label{SZEGOPIintroa}  
\hat{\Pi}(x,y) =  \int_{\R^+} e^{\sigma \h \psi(x,y)} s(x, y ,\sigma) d \sigma  + \hat{R}(x,y). 
\end{equation} 
where $\h \psi(x,y)$ is the almost-CR-analytic extension of $\h \psi(x,x)=-\rho(x) = \log \|x\|^2$, and $s(x,y,\sigma) = \sigma^m s_m(x,y) + \sigma^{m-1} s_{m-1}(x,y) + \cdots$ has a complete asymptotic expansion.  
In local trivialization \eqref{X-triv}, 
\[ \h \psi(x,y) = i (\theta_x - \theta_y) + \psi(z, w) - \half \varphi(z) - \half \varphi(w),\]
where $\psi(z,w)$ is the almost analytic extension of $\varphi(z)$.

\subsection{Lifting the Hamiltonian flow to a contact flow on $X_h$.}\label{LIFT} 
In this seection we review the definition of the  lifting of a Hamiltonian flow to a contact flow, following \cite[Section 3.1]{ZZ17}. 
Let $H: M \to \R$ be a Hamiltonian function on $(M, \omega)$. Let $\xi_H$ be the Hamiltonian vector field associated to $H$, such that $dH = \iota_{\xi_H}\omega$. 
The purpose of this section is to  lift $\xi_H$ to a contact vector field $\hat{\xi}_H$ on $X$. Let $\alpha$ denote the contact 1-form \eqref{alpha-def} on $X$, and $R$ the corresponding Reeb vector field determined by $\la \alpha, R \ra =1$ and $\iota_{R} d\alpha=0$. One can check that $R=\pa_\theta$. 

\begin{defin}
(1) The horizontal lift of $\xi_H$ is a vector field on $X$ denoted by  ${\xi}_H^h$. It is determined by 
\[ \pi_*{\xi}_H^h = \xi_H, \quad \la \alpha, \xi_H^h \ra = 0. \]
(2) The contact lift of $\xi_H$ is a vector field on $X$ denoted by  $\h {\xi}_H$. It is determined by
\[ \pi_*\h {\xi}_H = \xi_H, \quad \lcal_{\h \xi_H} \alpha = 0. \]
\end{defin}

\bl \label{xiHLEM} 
The contact lift $\h \xi_H$ is given by
\[ \hat{\xi}_H = {\xi}_H^h - H R. \]
\el

The Hamiltonian flow on $M$ generated by $\xi_H$ is denoted by $g^t$
\[ g^t: M\to M, \quad g^t = \exp(t \xi_H). \]
The contact flow on $X$ generated by $\h \xi_H$ is denoted by $\h g^t$
\[\h g^t: X \to X, \quad \h g^t = \exp(t \h \xi_H). \]

\begin{lem} \label{gtform} In  local trivialization \eqref{X-triv}, we have a useful formula for the flow, 
	  $\hat{g}^t$ has the form (see \cite[Lemma 3.2]{ZZ17}):
\[ \hat{g}^t(z, \theta) = (g^t(z), \;\; \theta 
	+   \int_0^t \half \la d^c \varphi, \xi_H \ra(g^s(z))ds  - t H(z)).  \]
\end{lem}

Since $\hat{g}^t$ preserves $\alpha$ it preserves the horizontal distribution $H(X_h) = \ker \alpha$, i.e.
\begin{equation} \label{HSPLIT} D \hat{g}^t: H(X)_x \to H(X)_{\hat{g}^t(x)}. \end{equation} It also preserves the vertical
(fiber) direction and therefore preserves the splitting $V \oplus H$ of $T X$. Its action in the vertical direction is determined by 
Lemma \ref{gtform}.  When $g^t$ is non-holomorphic, $\hat{g}^t$ is not CR holomorphic, i.e. does not preserve the horizontal complex structure $J$ or the
splitting of $H(X) \otimes \C$ into its $\pm i $ eigenspaces.

\subsection{Toeplitz Quantum Dynamics}
Here we consider quantization for the Hamiltonian flow $g^t$ on holomorphic sections of $L^k$, or CR-functions of degree $k$ on $X$. 
An operator $T: C^\infty(X) \to C^\infty(X)$ is called a {\em Toeplitz operator of order $k$}, denoted as $T \in \tcal^{k}$, if it can be written as $T = \h \Pi \circ Q \circ \h \Pi$, where $Q$ is a pseudo-diffferential operator on $X$ . Its principal symbol $\sigma(T)$ is the restriction of the principal symbol of $Q$ to the symplectic cone 
\[ \Sigma = \{(x, r \alpha(x)) \mid r > 0\} \cong X \times \R_+ \subset T^*X.\] 
The symbol satisfies the following properties
\[
\bcs
\sigma(T_1 T_2) = \sigma(T_1) \sigma(T_2); \\
\sigma([T_1, T_2]) = \{ \sigma(T_1), \sigma(T_2)\}; \\
\text{If $T \in \tcal^k$, and $\sigma(T) = 0$, then $T \in \tcal^{k-1}$.}
\ecs
\]
The choice of the pseudodifferential operator $Q$ in the definition of $T = \hPi \, Q \, \hPi$ is not unique. However, there exists some particularly nice choices. 
\bl[\cite{BG81} Proposition 2.13] \label{lm:niceQ}
Let $T$ be a Toeplitz operator on $\Sigma$ of order $p$, then there exists a pseudodifferential operator $Q$ of order $p$ on $X$, such that $[Q, \hPi] = 0$ and $T = \hPi \, Q \, \hPi$. 
\el

Now we specialize to the setup here, following closely \cite{RZ}. Consider an order one self-adjoint Toeplitz operator 
\[ T = \h \Pi \circ (H \cdot \mathbf{D}) \circ \h \Pi, \] where $\mathbf{D} = (-i \pa_\theta)$ and $\pa_\theta$ is the fiberwise rotation vector field on $X$, and $H$ is multiplication by $\pi^{-1}(H)$, where we abuse notation and identify $H$ downstairs with its pullback upstairs $\pi^{-1}(H)$. We note that $\mathbf{D}$ decompose $L^2(X)$ into eigenspaces $\oplus_{k \in \Z} L^2(X)_k$ with eigenvalue $k \in \Z$. 
The symbol of $T$ is a function on $\Sigma \cong X \times \R_+$, given by 
\[ \sigma(T) (x, r)=  (\sigma(H) \sigma(\mathbf{D})|_\Sigma)(x,r) = H(x) r, \quad \forall (x,r) \in \Sigma. \]
\bd[\cite{RZ}, Definition 5.1]
Let $\h U(t)$ denote the one-parameter subgroup of unitary operators on $L^2(X)$, given by 
\be \h U(t) := \hPi\, e^{i t \hPi (\mathbf{D} H) \hPi} \, \hPi : \hcal(X) \to \hcal(X), \ee
and let $\h U_k(t)$ \eqref{UkDEF} denote the Fourier component acting on $L^2(X)_k$:
\be \h U_k(t) :=  \hPi_k \, e^{i t \hPi (k H) \hPi} \, \hPi_k: \hcal_k(X) \to \hcal_k(X) \label{UkDEFb} \ee
We use $U_k(t)$ to denote the corresponding operator on $H^0(M, L^k)$. 
\ed

\bpp[\cite{RZ}, Proposition 5.2]
$\h U(t)$ is a group of Toeplitz Fourier integral operators on $L^2(X)$, whose underlying canonical relation is the graph of the time $t$ Hamiltonian flow of $r H$ on the sympletic cone $\Sigma$ of the contact manifold $(X,\alpha)$. 
\epp

\begin{prop} [\cite{Z97}] \label{SC}
	There exists a semi-classical symbol  $\sigma_{k}(t)$ so that the unitary group \eqref{UkDEFb}  has the form
	\be  \label{TREP}  \hat{U}_k(t)   = \hat{\Pi}_{k}  (\hat{g}^{-t})^* \sigma_{k}(t) \hat{\Pi}_{k}  \ee
	modulo smooth kernels of order $k^{-\infty}$.
\end{prop}

It follows from the above proposition and the   Boutet de Monvel--Sj\"ostand parametrix  construction that
$\h U_k(t, x, x)$ admits an oscillatory integral representation of the form,
\be \hat{U}_k (t, x, x)  \simeq \int_X \int_0^{\infty}    \int_0^{\infty} \int_{S^1}  \int_{S^1} 
e^{  \sigma_1 \h \psi(r_{\theta_1} x,  \h g^t y) + \sigma_2 \h \psi(r_{\theta_2} y, x) - i k \theta_1 -  i k \theta_2}  
S_k  d \theta_1 d \theta_2 d \sigma_1  d \sigma_2 d y  \label{hU} 
\ee
where $S_k$ is a semi-classical symbol, and the asymptotic symbol $\simeq$ means that the difference of the two sides is rapidly decaying in $k$.

\section{\label{BF} Bargmann-Fock space}
In this section, we illustrate the various definition of the background section using the example of Bargmann-Fock (BF) space. We also define the osculating BF space for at the tangent space $T_zM$ for a general \kahler manifold, and show that in the semi-classical limit as $k \to \infty$ the Bergman kernel near the diagonal reduces to the BF model at leading order. 

\subsection{Set-up}
Let $M=\C^m$ with coordinate $z_i=x_i + \sqrt{-1} y_i$, $L \to M$ be the trivial line bundle. We fix a trivialization and identify $L \cong \C^m \times \C$. 
We use  \kahler form $\omega = i \sum_i dz_i \wedge d\bar z_i$ and \kahler potential $\varphi(z)=|z|^2: = \sum_i |z_i|^2$. \footnote{Our choice of $\omega$ may differ from other conventions by factors of $2$ or $\pi$.}  The Bargmann-Fock space of degree $k$ on $\C^m$ is defined by
\[ \hcal_k = \{ f(z) e^{-k|z|^2/2} \mid f(z) \text{ holomorphic function on $\C^m$}, \quad \int_{\C^m} |f|^2 e^{-k|z|^2} < \infty \}. \]
The volume form on $\C^m$ is $d \Vol_{\C^m} = \omega^m/m!$. 

More generally, fix $(V, \omega)$ be a real $2m$ dimensional symplectic vector space. Let $J: V \to V$ be a $\omega$ compatible linear complex structure, that is $g(v,w): = \omega(v,Jw)$ is a positive-definite bilinear form and $\omega(v,w) = \omega(Jv, Jw)$. 
There exists a canonical identification of $V \cong \C^m$ up to $U(m)$ action, identifying $\omega$ and $J$. We denote the BF space for $(V, \omega, J)$ by $\hcal_{k,J}$. 

The circle bundle $\pi: X \to M$ can be trivialized as $X \cong \C^m \times S^1$. The contact form on $X$ is
\[ \alpha = d\theta + (i/2) \sum_j(z_j d\bar z_j - \bar z_j dz_j). \]
If $s(z)$ is a holomorphic function (section of $L^k$) on $\C^m$, then its CR-holomorphic lift to $X$ is 
\[ \h s(z, \theta) = e^{k(i \theta - \half |z|^2)} s(z). \]
Indeed, the horizontal lift of $\pa_{\bar z_j}$ is $ \pa_{\bar z_j}^h =\pa_{\bar z_j} -  \frac{i}{2} z_j \pa_\theta, $
and $\pa_{\bar z_j}^h \h s(z, \theta) = 0$. 
The volume form on $X=\C^m \times S^1$ is $d \Vol_{X} =(d\theta/2\pi) \wedge \omega^m/m!$.

\subsection{Bergman Kernel on Bargmann-Fock space}
 The  degree $k$  Bergman kernel downstairs on $\C^m$ is given by
\[ \Pi_k(z,w) = \kk^{m} e^{z \bar w}. \]
Given any function $f \in L^2(\C^m, e^{-k|z|^2/2}dVol_{\C^m})$, its orthogonal projection to holomorphic function is given by 
\[ (\Pi_k f)(z) = \int_{\C^m} \Pi_k(z,w) f(w) e^{-k|w|^2} d \Vol_{\C^m}(w). \]

The degree $k$ Bergman (\Szego) kernel $\h \Pi_k(\h z, \h w)$ upstairs for $X=\C^m \times S^1$ is given by 
\begin{equation} \label{BFSZEGODEF}   \h \Pi_k(\h z, \h w) =   \kk^m e^{k \h \psi(\h z, \h w)}, \end{equation} where $\h z =(z, \theta_z), \; \h w = (w, \theta_w)$ and the phase function is 
\be \label{BF-phase}  \psi(\h z, \h w)  = i   (\theta_z - \theta_w) + z \bar w - \half |z|^2 - \half |w|^2. \ee

\subsection{Heisenberg Representation}
The space $\C^m \times S^1$ can be identified with the reduced Heisenberg group $\Hb^m_{red}$, where the group multiplication is given by 
\[ (z, \theta) \circ (z', \theta') = (z+z', \theta+\theta' + \Im( z \bar z')). \]

\bl
The contact form $\alpha = d\theta + \frac{i}{2} \sum_j (z_j d\bar z_j - \bar z_j dz_j)$ on $\H^m_{red}$ is invariant under the left multiplication
\[ L_{(z_0, \theta_0)}: (z, \theta) \mapsto (z_0, \theta_0)  \circ (z, \theta) = (z + z_0, \theta + \theta_0 + \frac{ z_0 \bar z-  \bar z_0 z }{2i} ). \]
\el
\bpf
\[ (L_{(z_0, \theta_0)}^* \alpha)|_{(z, \theta)} = d (\theta + \theta_0 + \frac{\bar z z_0 - \bar z_0 z}{2i} ) + \frac{i}{2} \sum_j ((z_j + z_{0j}) d \bar z_j - (\bar z_j+\bar z_{0j}) d z_j) = \alpha|_{(z, \theta)}.\]
\epf

In particular, $\Hb^m_{red}$ preserves the volume form $\alpha \wedge (d\alpha)^m/m!$ on $X$, hence defines a unitary operator acting on the degree $k$ CR functions on $X$. 

The infinitesimal Heisenberg group action on $X$ can be identified with contact vector field generated by a linear Hamiltonian function $H: \C^m \to \R$. 
\bl \cite[Section 3.2]{ZZ17}\label{flow-lin}
For any $\beta \in \C^m$, we define a linear Hamiltonian function on $\C^m$ by
\[ H(z) = z \bar \beta + \beta \bar z. \]
The Hamiltonian vector field on $\C^m$ is  
\[ \xi_H = - i \beta \pa_z + i \bar \beta \pa_{\bar z}, \]
and its contact lift is 
\[ \h \xi_H = - i \beta \pa_z + i \bar \beta \pa_{\bar z} -\half( z \bar \beta + \beta \bar z) \pa_\theta. \]
The time $t$ flow $\h g^t$ on $X$ is given by left multiplication 
\[ \h g^t(z,\theta) = (-i \beta t, 0) \circ (z,\theta) = (z-i\beta t, \theta - t\Re(\beta \bar z)). \]
\el

\subsection{\label{METASECT} Metapletic Representation}
Let $\R^{2m}, \omega = 2 \sum_{j=1}^m dx_j \wedge d y_j$ be a sympletic vector space. The space $Sp(m, \R)$ consists of linear transformation $S: \R^{2m} \to \R^{2m}$, such that $S^*\omega = \omega$. In coordinates, we write 
\[ \bma x' \\ y' \ema = S \bma x \\ y \ema = \bma A & B \\ C & D \ema \bma x \\ y \ema. \]
In complex coordinates $z_i = x_i + i y_i$, we have then 
\[ \bma z' \\ \bar z' \ema = \bma P & Q \\ \bar Q & \bar P \ema \bma z \\ \bar z \ema =: \acal \bma z \\ \bar z \ema, \]
where 
\begin{equation} \label{PQDEF}  \bma P & Q \\ \bar Q & \bar P \ema = \wcal^{-1} \bma A & B \\ C & D \ema \wcal, \quad \wcal = \frac{1}{\sqrt 2} \bma I & I \\ -i I & iI \ema. \end{equation}
The choice of normalization of $\wcal$ is such that $W^{-1} = W^*$.Thus, 
\[ P = \half(A+D + i (C-B)). \]
 We say such $ \acal \in Sp_c(m, \R) \subset M(2n,\C)$. The following identities are often useful.
\bpp [ \cite{F89} Prop 4.17]
Let $ \acal= \bma P & Q \\ \bar Q & \bar P \ema \in Sp_c$, then \\
(1) $ \bma P & Q \\ \bar Q & \bar P \ema^{-1} =\bma P^* & -Q^t \\ -Q^* & P^t \ema = K  \acal^* K$, where $K =  \bma I & 0 \\ 0 & -I \ema.$ \\
(2) $ PP^* - QQ^* = I$ and $P Q^t = Q P^t$. \\
(3) $P^*P - Q^t \bar Q = I$ and $P^t \bar Q = Q^* P$. 
\epp

The (double cover) of $Sp(m,\R)$ acts on the (downstairs) BF space $\hcal_k$ via kernel: given $M= \bma P & Q \\ \bar Q & \bar P \ema \in Sp_c$, we have
\[  \kcal_{k,M}(z,  w) = \kk^{m} (\det P)^{-1/2} \exp \left\{k \half \left( z \bar{Q} P^{-1} z + 2 \bar{w} {P}^{-1} z
- \bar{w} P^{-1} Q \bar w \right)  \right\} \]
where the ambiguity of the sign the square root $(\det P)^{-1/2}$ is determined by the lift to the double cover. When $ \acal=Id$, then $\kcal_{k, \acal}(z, \bar w) = \Pi_k(z, \bar w)$. Similarly, we have the kernel upstairs on $X$ as 
\be \h \kcal_{k, \acal}(\h z, \h w) = \kcal_{k,M}(z, \bar w) e^{k(i\theta_z -|z|^2/2) + k(-i\theta_w - |w|^2/2)}. \label{hatK}\ee

A quadratic Hamiltonian function $H: \C^m \to \R$ will generates a one-parameter family of symplectic linear transformations $ \acal_t = g^t: \C^m \to \C^m$. However, $ \acal_t$ is only $\R$-linear but not $\C$-linear, i.e. $M_t$ does not preserve the complex structure of $\C^m$. Hence, one need to orthogonal project back to holomorphic sections. To compensate for the loss of norm due to the projection, one need to multiply a factor $\eta_{ \acal_t}$. This is in the spirit of Proposition \ref{SC}. 

\bpp \label{toep-met}
Let $ \acal: \C^m \to \C^m$ be a linear symplectic map, $\acal =  \bma P & Q \\ \bar Q & \bar P \ema$, and let $\h  \acal: X \to X$ be the contact lift that fixes the fiber over $0$, then 
\[  \h \kcal_{k, \acal}(\h z, \h w) = (\det P^*)^{1/2} \int_X \h \Pi_k(\h z, \h  \acal \h u) \h \Pi_k(\h u, \h w) d \Vol_X(\h u) \]
\epp
\bpf
The contact lift $\h \acal: \C^m \times S^1 \to \C^m \times S^1$ is given by $\acal$ acting on the first factor:
\[ \h \acal: (z, \theta) \mapsto (P z + Q \bar z, \theta), \]
one can check that $\h \acal^* \alpha = \alpha$. The integral over $X$ is a standard complex Gaussian integral, analogous to \cite[Prop 4.31]{F89}, and with determinant Hessian $1/|\det P|$, hence we have $(\det P^*)^{1/2}/|\det P| = (\det P)^{-1/2}$. 
\epf

\subsection{\label{TOEPMETA}Toeplitz construction of the metaplectic representation}

  As in \cite{Dau80},  the metaplectic representation  $W_J(S)$ 
of  $S \in Mp(n,\R)$ on $\hcal_J$  can also be constructed by the Toeplitz approach.  First, let
$U_S$ be the unitary translation operator  on $L^2(\R^{2n}, d L)$ defined  by $U_S F(x, \xi): = F(S^{-1}(x, \xi))$. The metaplectic representation of $S$ on $\hcal_J$
is given by (\cite{Dau80},(5.5) and (6.3 b)) 
\begin{equation} \label{eta}W_J(S) = \eta_{J,S} \Pi_J U_S \Pi_J, \\
\end{equation} where  we define (see 
  \cite{Dau80} (6.1) and (6.3a)), 
\begin{equation} \label{ETAJS} \begin{array}{lll}
\eta_{J,S} & = & 2^{-n}  \det (I - i J) + S (I + i J)^{\half}

 \end{array} \end{equation} and $\Pi_J$ is the Bargmann-Fock
Szeg\"o projector \eqref{PIBF}.  

Also define  $\beta_{J, SJS^{-1}}  =  2^{-n/2} [\det (SJ + JS) ]^{1/4}$.
Then,
$|\eta_{J,S}| = \beta_{J, SJS^{-1}}$. In fact (see \cite{Dau80}, above (6.3a), and (B6))
$$|2^{-n}  \det (I - i J) + S (I + i J)^{\half}| = [\det (SJ + JS) ]^{1/2} = 2^n \beta^2_{J, SJS^{-1}}. $$

We further record the identities,
\begin{equation}
\det (SJ + JS) = \det (I + J^{-1} S^{-1} J S) = \det (I + S^* S).
\end{equation}

The following identity gives another explanation of the presence
of $(\det P_n)^{-\half}$  in \eqref{gcalndef}.

\begin{lem} \label{KEYID2} (see \cite{Dau80}, p. 1388)

\begin{equation} \begin{array}{lll} \eta_{JS} \beta^{-2}_{J, SJS^{-1} } 
& = & (\eta_{J S})^{*-1} = \eta_{JS}\; 2^n  (\det (I + S^* S))^{-\half} 
 \end{array}\end{equation}
 and
(cf. \cite{Dau80}, p. 1388,
\begin{equation}\label{DAUB}  \begin{array}{l}(\eta^*_{J, S})^{-1} = \det ((I + i J) + S (I - i J)) = 2^n \det(A + D + i (B - C)) = \det P^*.
  \end{array}  \end{equation}

\end{lem}

\begin{proof} The first equality is proved on p. 1388 of \cite{Dau80}. The second asserts that
\begin{equation} \label{BETA} \beta_{J, SJS^{-1} } = 2^{-n/2} (\det (I + S^* S))^{\frac{1}{4}}, \end{equation}
which follows from \eqref{ETAJS} and identity (ii) above.

\end{proof}

\begin{cor} \label{Uinv}$\eta_{J, U S U^{-1}} = \eta_{J, S}$ where $U \in U(m)$. \end{cor}

\begin{proof} This follows  from replacing $S$ by $USU^{-1}$ and using that $U J = J U$. \end{proof}

\subsection{\label{OBFSECT} Osculating Bargmann-Fock space}

In this subsection, we first define the osculation Bargmann-Fock space for any fixed point $z \in M$, using the triple $(T_z M, \omega_z, J_z)$. Then, we define the preferred local coordinates in a neighborhood $U$ of $z$ and a preferred frame section $e_L$ of $L$ over $U$, which together determines a coordinate system of the circle bundle $X|_U$ over $U$. In these special coordinate, the Boutet-\Sj  phase can be approximated by the Bargmann-Fock-Heisenberg phase function modulo cubic order terms.

\begin{defin} \label{OSCBFDEF} 
	
	Given a  point $x \in X_h$ (resp. $z \in M$), we define the {\it osculating Bargmann-Fock
		space} at $x$ (resp. $z$) to be the Bargmann-Fock space of $(H_{x}X, J_{x}, \omega_{x})$ resp. $(T_z M, J_z, \omega_z)$. We denote 
	it by  $\hcal_{J_{x}, \omega_{x}}$ (resp. $\hcal_{J_z, \omega_z}$). \end{defin}

If $z$  is a periodic point for $g^t$, 
let $\gamma = \bigcup_{0 \leq s \leq t} g^s z$ be the corresponding
closed geodesic, 
and we may apply the metaplectic representation to define $W_{J_z}(D g^t|_z)$
as a unitary operator on $\hcal_{J_z, \omega_z}$.   There is a square root ambiguity
which can be resolved as in \cite{Dau80} but for our purposes it is not very important and for brevity we omit it from the discussion.

\begin{defin}\label{KDEF}  Let $p \in M$.  A coordinate system $(z_1,\dots,z_m)$ on a
	neighborhood $U$ of $p$ is called {\it K-coordinates\/} at $p$ if
	$$i \sum_{j=1}^m d z_j \wedge d\bar z_j =\omega|_{p} \,.$$
	Let $e_L$ be a local frame and let $\phi(z) = - \log ||e_L(z)||^2_{h}$, 
	if in a K-coordinates
	\begin{equation} \label{phiK} \phi(z) = |z|^2 + \sum_{ J K} a_{ J K} z^J \bar{z}^K, \;\;{\rm with}\; |J| \geq 2, |K| \geq 2. \end{equation}
	then $e_L$ is called a K-frame. 
\end{defin}

K-coordinates are defined by Lu-Shiffman in Definition 2.6 of \cite{LuSh15}. Existence of K-coordinates and K-frames are proved in \cite{LuSh15} (Lemma 2.7). Further, in K-coordinates,
\begin{equation} \label{omega} \omega = \omega_0 + \sum_{i j k \ell} R_{i j k \ell} z_i \bar{z}_j dz_k \wedge d \bar{z}_{\ell} + \cdots,
\;\; \omega_0 = \sum_j dz_j \wedge d \overline{z_j}. 
\end{equation}

The K-frame and K-coordinates together give us `Heisenberg
coordinates':

\begin{defin} A  {\it Heisenberg coordinate chart\/} at a point $x_0$ in
	the principal bundle $X$ is a coordinate chart
	$\rho:U \to
	V$ with $0\in U\subset \C^m\times S^1$ and $\rho(0)=x_0\in V\subset X$ of the
	form
	\begin{equation}\rho(z_1,\dots,z_m,\theta)= e^{i\theta} \frac{
		e^*_L(z)}{||e^*_L(z)||_{h^k}} \,,\label{coordinates}\end{equation} where
	$e_L$ is a preferred local frame for $L\to M$ at $P_0=\pi(x_0)$, and
	$(z_1,\dots,z_m)$ are K-coordinates centered at $P_0$.
	(Note that $P_0$ has coordinates $(0,\dots,0)$ and $e_L^*(P_0)=x_0$.)
\end{defin}

In these coordinates, the Boutet-\Sj phase $\psi(x,y)$ may be approximated 
modulo cubic remainder terms by
the  Bargmann-Fock-Heisenberg  phase \eqref{BF-phase}.

The lifted \szego kernel is shown in \cite{ShZ02}   and in Theorem 2.3 of \cite{LuSh15} to have the scaling asymptotics,
\begin{theo} \label{SHLU} Let $P_0\in M$ and choose a Heisenberg coordinate chart about $P_0$.
	\[k^{-m} \h \Pi_{h^k}(\frac{u}{\sqrt{k}}, \frac{\theta_1}{k}, \frac{v}{\sqrt{k}}, \frac{\theta_2}{k}) = \hat{\Pi}^{T_z M}_{h_z, J_z}(u, \theta_1,  v, \theta_2) \left(1 + k^{-1} A_1(u, v, \theta_1, \theta_2) + \cdots \right),
	\]
	where $\Pi^{T_z M}_{h_z, J_z}$ is the osculating Bargman-Fock \szego kernel  for $k=1$ and for the tangent space $T_z M \simeq \C^m$
	equipped with the complex structure $J_z$ and Hermitian metric $h_z$. 
\end{theo}
Here we identify the coordinates $(u, \theta_1, v, \theta_2)$ with linear coordinates on $T_z M \times S^1 \times T_z M
\times S^1$.

\section{ Proof of Theorem \ref {PikfTH}  }

In this section we study the rescaled Weyl sum
  \[ \Pi_{k,f}^E(z,z) :=  \sum_{j} f(k(\mu_{k,j} - E)) \Pi_{k,j}(z,z). \]
  Our purpose is to prove Theorem \ref{PikfTH} . By comparison with interface
  asymptotics \cite{ZZ17}, we now need to consider the Hamiltonian flow for long
  times. 
  
  The main idea of the proof is that aside from the holonomy factor (the value of the phase at the critcal point), the data of the principal term  in Theorem \ref{PikfTH} localizes at the periodic point. That is, the data come from the derivative of the first return map and do not involve data along the orbit. Too see this,   
we use  the quadratic Taylor  approximation of the phase $\psi(x, \hat{g}^t y) +\psi( y, x)$ in $(t, y)$  around a periodic point $(T, x)$. First, we approximate the phase $\psi$ by its osculating Bargmann-Fock approximation  $\psi_0$ at $x$. Further we approximate $\hat{g}^t$ by its linear approximation   $D \hat{g}^t$. 
We also need to  determine the quadratic appoximation to 
 the holonomy term of the phase coming from the $\theta$ variable. This part of the calculation is apriori non-local. But we  show in Proposition \ref{phase-quad} 
that the Hessian of the holonomy term $\hat{\theta}_w(T)$  vanishes at the periodic point. After these Taylor approximations, the calculation is essentially reduced to the linear Bargmann-Fock case of Section \ref{BF}.

\subsection{\label{MAINPROOFSECT} Stationary Phase Integral expression}
  Let $z \in M$ and $x \in X$ such that $\pi(x)=z$. Let $f \in \scal(\R)$ with Fourier transform $\h f(t) = \int f(x) e^{i t x} \frac{dx}{2\pi}$ compactly supported. We combine the definition \eqref{UktPARA}  with two compositions of the Boutet de Monvel-Sjoestrand parametrix \eqref{SZEGOPIintroa} to get
 \bea
  \Pi^E_{k, f}(z) &=& \int_{\R} \hat{f}(t) e^{- it k E} \hat{U}_k(t, x,x ) dt \\
   &=&  \int_{\R}\int_X \int_{S^1}\int_{S^1}\int_{\R_+}\int_{\R_+} \hat{f}(t) e^{k \Psi(t, x, y, \sigma_1, \sigma_2, \theta_1, \theta_2)} A_k d\sigma_1 d\sigma_2 d \theta_1 d\theta_2 dy dt + O(k^{-\infty}). \label{int-1}
  \eea
  where the phase function is given by,
  \be \Psi(t, x, y, \sigma_1, \sigma_2, \theta_1, \theta_2) = -i t E + \sigma_1 \h \psi(r_{\theta_1}x, \h g^t y) +  \sigma_2 \h \psi(r_{\theta_2}y,x) - i \theta_1 - i \theta_2
  \label{phase}
  \ee
  and $A_k$ is a semi-classical symbol. We consider the critical points and the determinant of the Hessian matrix of the phase. 
 
  We will work with a K-coordinate and K-frame in a neighborhood $U$ of $z$. In this coordinate, $z = (0,...,0) \in \C^m$, $x = (0,\cdots, 0; 0) \in \C^m \times S^1$, and $y = (w; \theta_w) \in \C^m \times S^1$. We denote $\h g^t y = (w(t); \theta_w(t))$. Since $\theta_w(t)-\theta_w$ only depends on $w,t$ but independent of $\theta_w$, then
  we define the holonomy phase for flow $\h g^t$: 
  \be \h \theta_w(t) := \theta_w(t)-\theta_w. \label{hol-cont}\ee
  Similarly, the holonomy phase $ \theta^h_w(t)$ for the horizontal flow $\exp(t \xi_H^h)$ is denoted by
  \be \exp(t \xi_H^h)(w; \theta_w) = (g^t w; \theta_w + \theta_w^h(t)). \label{hol-hor}\ee
  Note that $\h \theta_w(t)$ depends on $H$, where as $\theta_w(t)$ only depend on $H$ modulo constant,or $dH$.

  \bpp
  Fix a K-coordinate and K-frame in a neighborhood $U$ at $z$. 
  Let $\chi: M \to \R$ be a smooth cut-off function supported in $U$ and constant equals to one near $z$. Then we have
  \be \Pi_{k,f}^E(z) = \int_{\R}\int_M \int_{S^1}\int_{S^1}\int_{\R_+}\int_{\R_+} \hat{f}(t) e^{k \Psi'(t,w,\sigma_1,\sigma_2,\theta_1,\theta_2) }  \chi(g^t w) \chi(w) S_kd\sigma_1 d\sigma_2 d \theta_1 d\theta_2 dw dt + O(k^{-\infty}).  \label{int-2}\ee
  where
   \be \Psi'(t,w,\sigma_1,\sigma_2,\theta_1,\theta_2) = -i t E + \sigma_1(i  \theta_1 -  i \h \theta_w(t) - \varphi(w(t)) +  \sigma_2 (i  \theta_2 - \varphi(w) ) - i  \theta_1 - i \theta_2. \label{phase2} \ee
   
  \epp
  \bpf
  Introducing the cut-off function $\chi$ in the integral \eqref{int-1} only changes the integral by $O(k^{-\infty})$. Within the support of the cut-off function, we may use the K-coordinates. 
  
  Then phase function $\Psi$ can be written as (within the coordinate patch): 
  \bea \Psi &=& -i t E + \sigma_1 (i \theta_1 - i \h \theta_w(t) - i \theta_w + \psi(0,w(t)) - \varphi(w(t))  \\
  && + \sigma_2 (i \theta_2 + i \theta_w + \psi(w,0) - \varphi(w) ) - i \theta_1 - i \theta_2 \\
  &=& -i t E + \sigma_1(i \wt \theta_1 -  i \h \theta_w(t) - \varphi(w(t)) +  \sigma_2 (i \wt \theta_2 - \varphi(w) ) - i \wt \theta_1 - i \wt \theta_2 
  \eea
  where $\wt \theta_1 = \theta_1 - \theta_w$ and $\wt \theta_2 = \theta_2+\theta_w$. We note $\psi(0,w)=0$ due to the choice of K-frame \eqref{phiK}.  After the change of variables, we see the phase $\Psi$ does not depend on $\theta_w$. Hence we may perform the $\theta_w$ integral, and rewrite $\wt \theta_i$ as $\theta_i$, to get the reduced phase function $\Psi'$. 
   \epf
    
%


\bpp \label{crit}
The critical points for $\Psi'$ \eqref{phase2} are as following: \\
(1) If $z \notin H^{-1}(E)$, there is no critical points. \\
(2) If $z \in H^{-1}(E)$ but $z \notin \pcal_E$, then the only critical point corresponds to $t=0$. \\
(3) If $z \in H^{-1}(E)$ and $z \in \pcal_E$, then for each $n \in \Z$, there is a critical point with $t = n T_z$, where $T_z$ is the primitive period of $g^t$ at z. 
\epp
\bpf
We will prove that the critical points for $\Psi'$ \eqref{phase} are given by
\[ w=0,\; w(t)=0,\; \sigma_1=\sigma_2=1,\; \theta_1=\h \theta_0(t),\; \theta_2 = 0.   \]

Taking derivatives of $\sigma_1$ and $\sigma_2$, we need to have
\[  i  \theta_1 -  i \h \theta_w(t) - \varphi(w(t) =0, \quad  i  \theta_2 - \varphi(w)  = 0. \]
Hence
\be \theta_1=\h \theta_0(t),\; \theta_2 = 0,  \label{sigma-cond} \ee
Thus, we may work in a neighborhood of $x$ from now on.

Taking derivatives in $\theta_1$ and $\theta_2$ and setting them to zero, we get
\[ \sigma_1 = 1, \quad \sigma_2 = 1.\]

Taking derivative in $t$ and setting it to zero, we have 
\[ \frac{\pa \Psi'}{\pa t} = -i E + i \sigma_1 \frac{d \h \theta_w(t)}{d t} = -i (E - \sigma_1 H(0)).  \] 
Thus, using $\sigma_1=1$, we have
$E = H(0).$

Finally, taking derivatives  in $w$, we have
\[ \frac{\pa \Psi'}{\pa w} = -i \sigma_1 \pa_w \h \theta_w(t) = - i \pa_w \theta_w(T) \]
where $T$ is a period. Since $\h g^T$ preserves horizontal space, and $\pa_w$ is in the horizontal space at $x=(0;0)$, hence 
\[ \pa_w \theta_w(T) = \la \alpha|_x,  (\h g^T|_x)_* \pa_w \ra = \la \alpha|_x,  \pa_w \ra = \la d\theta,  \pa_w \ra = 0. \]
\epf

\subsection{Determinant of Hessian of $\Psi'$}
Let $T$ be a period of $g^t$ at $z$ (possibly zero). To compute the contribution at $t=T$, we will do a slight change of variables. 
\bl
Define new integration variables
\[ t = T+ t', \;w = g^{-t'} w', \; \theta_1 = \theta_1' - \h \theta_{w'}(-t')  , \; \theta_2 = \theta_2' + \h \theta_{w'}(-t').\]
Then the Jacobian factor is $1$, and the 
phase function $\Psi_T$ in the new variables is
\be \label{phase3} \Psi_T(t', w', \sigma_i,\theta_i')  = -i (T+t') E +   \sigma_1(i  \theta_1' -  i \h \theta_{w'}(T) - \varphi(w'(T)) +  \sigma_2 (i  \theta_2' + \h \theta_{w'}(-t') - \varphi(w'(-t')) ) - i  \theta_1' - i \theta_2'.\ee
(We will drop the prime from now on.)
\el
\bpf
The Jacobian matrix is block-upper-triangular, with the $w-w'$ block having determinant $1$, 
since $g^t$ preserves the volume form.

The holonomy for flow $\h g^t$ can be written as 
\[ \h \theta_w(t) = \theta_w(t) - \theta_w(0) = \theta_{w'}(T) - \theta_{w'}(-t') = \h \theta_{w'}(T) - \h \theta_{w'}(-t'). \] 
\epf

\begin{lem}
The Hessian matrix for $\Psi_T(t,w,\sigma_i, \theta_i)$ at $t=0,w=0, \sigma_i=1, \theta_1=\h \theta_0(T), \theta_2=0$ is as
\be
Hess\Psi_T =  \kbordermatrix{ & \sigma_1 & \theta_1 & \sigma_2 &\theta_2 & t & w \\
	\sigma_1 & 0 & i & 0 & 0 & 0 & 0 \\
	\theta_1 & i & 0 & 0 & 0 & 0 & 0 \\
	\sigma_2 & 0 & 0 & 0 & i & 0 & 0 \\
	\theta_2 & 0 & 0 & i & 0 & 0 & 0 \\
	t & 0 & 0 & 0 & 0 & \pa_{tt}\Psi_{T} & \pa_{tw}\Psi_{T} \\
	w & 0 & 0 & 0 & 0 & \pa_{wt}\Psi_{T} & \pa_{ww}\Psi_{T} }.
\ee 
In particular, at this critical point, we have
\[ \det Hess\Psi_T = \det \bma  \pa_{tt}\Psi_{T} & \pa_{tw}\Psi_{T} \\\pa_{wt}\Psi_{T} & \pa_{ww}\Psi_{T}\ema. \]
\end{lem}
\bpf
The calculation is very similar to that in the proof of Proposition \ref{crit}, and is therefore omitted. 
\epf

\subsection{Quadratic Approximation to the Phase}
To compute the Hessian of the phase function $\Psi_T$ in $t$ and $w$, suffice to set $\sigma_i,\theta_i$ to their critical value, and compute the Taylor expansion of $\Psi_T$ to second order. Thus, we get
\[ \Psi_T'(t,w) := -i (T+t) E  -  i \h \theta_{w}(T) - \varphi(w(T)) +   i \h \theta_{w}(-t) - \varphi(w(-t)).  \]
We will consider second order Taylor expansion in each term. We write $\simeq$ for equal modulo cubic order term.

Suppose $H$ has Taylor expansion 
\[ H(w) = E + (\alpha \bar{w} + w \bar \alpha) + O(|w|^2). \]
We define the corresponding $H_{BF}$ for the osculating BF space $\C^m \cong T_z M$, as the linear term of $H$:
\[ H_{BF}(w) = \alpha \bar{w} + w \bar \alpha. \]
We denote the BF model potential as $\varphi_{BF}(z) = |z|^2$. Let $\h g^t_{BF}$ be the flow generated by $H_{BF}$ on $X_{BF}=\C^m \times S^1$, such that
\[ \h g^t_{BF}(w; \theta_w) = (w(t)_{BF}; \theta_w + \h \theta_w(t)_{BF}). \]
Then, we have the following comparison result
\bpp
(1) $\h \theta_{w}(-t) -  t E = \h \theta_{w}(-t)_{BF} + O_3 = \half(\alpha \bar{z} + z \bar \alpha) t $. \\
(2) $\varphi(w(T)) = |D g^T w|^2+ O_3$  \\
(3) $\varphi(w(-t)) = |w(-t)_{BF}|^2+ O_3 = |w + i \alpha t|^2 + O_3$
\epp
\bpf
(1) $\h \theta_{w}(-t) = \int_{0}^t \half d^c \varphi (\xi_H) |_{w(s)} ds+ tH(w)$. Since $d^c\varphi|_{w} = O(|w|)$ and the integral interval is first order in $t$, hence
\[ \int_{0}^t \half d^c \varphi (\xi_H) |_{w(s)} ds = t \half d^c \varphi (\xi_H) |_{w} + O_3 = t \la \half d^c \varphi|_w, \xi_H|_0 \ra + O_3 = \int_{0}^t \half d^c \varphi_{BF} (\xi_{H_{BF}}) |_{w(s)} ds + O_3.  \]
And $tH(w) = t(E+H_{BF}(w)) + O_3$. Hence 
\[ \h \theta_{w}(-t) -  t E = \int_{0}^t \half d^c \varphi_{BF} (\xi_{H_{BF}}) |_{w(s)} ds + t(E+H_{BF}(w)) - tE + O_3 = \h \theta_{w}(-t)_{BF} + O_3. \]
Finally, we may use Lemma \ref{flow-lin} to compute the increment in $\theta$. 

(2) Since $\varphi(w) = |w|^2 + O(|w|^3)$ and $w(T) = g^T(w) = g^T(0) + D g^T w + O(|w|^2) = D g^T w + O(|w|^2)$, hence 
\[\varphi(w(T)) = |D g^T w|^2 + O_3 \]

(3) Since $\xi_H =  -i\alpha \pa_z + i \bar \alpha \pa_{\bar z}  + O(|z|)$, we have $w(-t) = w + i \alpha t + O_2$, hence
\[ \varphi(w(-t)) = |w + i \alpha t|^2 + O_3 = |w(-t)_{BF}|^2+ O_3. \]
\epf

\bpp \label{phase-quad}
\[ \h \theta_{w}(T) = \h \theta_{0}(T) + O(|w|^3). \]
\epp
\bpf
The  proof is rather long, so we break it up into the following two Lemmas.

\bl \label{phase-local}
There exists a neighborhood $V \subset U$ of $z$, such that for any $w \in V$, and any path  
$\gamma:[0,1] \to V$ from $z$ to $w$, we have 
\[  \h \theta_w(T) =\h \theta_{0}(T) -  \int_{\gamma} \half d^c \varphi + \int_{g^T(\gamma)} \half d^c \varphi.  \]
\el
\bpf
We only give proof for $T=n T_z$, $n>0$, the $n \leq 0$ case is analogous. 
Let $\{(U_i, e_i, \varphi_i)\}_{i=1}^n$ be a sequence of coordinate patch $U_i$, such that 
there exists a partition of $[0,T]$: $0=t_0 < t_1 < \cdots < t_n=T$, such that $U_i$ covers the $i$-th segment of the orbit  $O_i= \{g^s z \mid t_{i-1} \leq s \leq t_i \}$, and
$e_i \in \Gamma(U_i,L)$ are non-vanishing holomorphic sections, and
$e^{-\varphi_i} = \|e_i\|^2$. Without loss of generality, we may take $U_1=U$. We identify index $n+i$ with $i$. 

Since $g^{t_{i}} z \in U_i \cap U_{i+1}$ for $0 \leq i \leq n$, hence 
\[  z \in V := \bigcap_{i=0}^n g^{-t_{i}}  (U_i \cap U_{i+1}). \]
For any $w \in V$, let $\gamma:[0,1] \to V$ be a path from $z$ to $w$. Let 
\[  \gamma_0 = \gamma, \quad \gamma_i = g^{t_i} \gamma. \] Then 
\[ Im(\gamma_i) \subset U_i \cap U_{i+1}, \forall 0 \leq i \leq n.  \]

Over $U_i \cap U_{i+1}$, define transition function $g_i = \log(e_{i+1} / e_i)$, such that $g_i = a_i + \sqrt{-1} b_i$, with $b_i(g^{t_i} z) \in [0,2\pi)$. Then we have 
\[ \|e_{i+1}\| = |g_i| \|e_i \| \RA e^{-\half \varphi_{i+1}} = e^{a_i} e^{-\half \varphi_i}  \RA \varphi_{i+1} - \varphi_i = -2 a_i. \]

Over $U_i$, 
let $\theta_i = e_i^* / \|e_i^*\|$ be the section in the co-circle bundle $X$. Then over $U_i \cap U_{i+1}$, we have 
\[  \log(e^*_{i+1} / e_i^*) = 1/g_i = e^{-a_i - \sqrt{-1} b_i} \RA \theta_{i+1}  -  \theta_i \equiv - b_i \mod 2\pi. \]
where we used additive notation for section valued in $S^1$. 

Then, the holonomy can be expressed using Lemma \ref{gtform} in each coordinate patch $U_i$ 
\[  \h \theta_w(T) =\theta_w(T) - \theta_w = \sum_{i=1}^n \int_{t_{i-1}}^{t_i} \half \la d^c \varphi_i, \xi_H \ra|_{ g^{s} w} ds  - (t_{i+1} - t_i)H(w) + b_i(g^{t_i} w). \]
Thus, we may take the difference
\bea
\h \theta_w(T) - \h \theta_{0}(T) &=& \sum_{i=1}^n \int_{t_{i-1}}^{t_i}  \half \la d^c \varphi_i, \xi_H \ra |_{ g^{s} w} ds - \int_{t_{i-1}}^{t_i} \half \la d^c \varphi_i, \xi_H \ra|_{ g^{s}z} ds  - (t_{i+1} - t_i)(H(w) -H(z)) \\
&&+ \sum_{i=1}^n b_i(g^{t_i} w) - b_i(g^{t_i}z) \\ 
&=& \sum_{i=1}^n \int_{0}^{1} \int_{t_{i-1}}^{t_i} \omega(\pa_t, \pa_s) dt ds - (t_{i+1} - t_i)(H(w) - H(z)) \\
&& +\sum_{i=1}^{n} - \int_{\gamma_{i-1}} \half d^c \varphi_i +  \int_{\gamma_{i}} \half d^c \varphi_i 
+ \sum_{i=1}^n \int_{\gamma_i} db_i \\ 
&=& \sum_{i=1}^n \int_{0}^{1} \int_{t_{i-1}}^{t_i} dH(\pa_s) dt ds - (t_{i+1} - t_i)(H(w) - H(z)) \\
&& - \int_{\gamma_0} \half d^c \varphi_1 + \sum_{i=1}^{n} \int_{\gamma_i} \half d^c(\varphi_i - \varphi_{i+1}) + \int_{\gamma_n} \half d^c \varphi_{n+1} + \sum_{i=1}^n \int_{\gamma_i} db_i \\ 
&=&  - \int_{\gamma_0} \half d^c \varphi_1 +  \int_{\gamma_n} \half d^c \varphi_{n+1} + \sum_{i=1}^{n} \int_{\gamma_i}  (d^c a_i + d b_i) \\
&=& - \int_{\gamma_0} \half d^c \varphi_1 +  \int_{\gamma_n} \half d^c \varphi_{1}
\eea
where in the last step, we used 
\[ d^c(a_i + \sqrt{-1} b_i) = d(\sqrt{-1} a_i - b_i) \RA d^c a_i = - d b_i. \]
\epf

\bl
For any fixed path $\gamma:[0,1] \to U$ starting from $0$, and for any $1 \gg \epsilon>0$, we have
\[ \int_{\gamma([0,\epsilon])} d^c \varphi = \int_0^\epsilon \la d^c \varphi, \dot \gamma(s)\ra ds = O(\epsilon^3) \]
\el
\bpf
If a path $\gamma:[0,1] \to U$ with $\gamma(0)=0$ and $\gamma(1)=w$ is a straight-line, then 
\[
\int_{\gamma} d^c \varphi = O(|w|^3).
\]
Indeed, consider the Taylor expansion of $\varphi(z)$ at $z=0$, 
\[ \varphi(z) = |z|^2 + O(|z|^3) \]
then
\[ d^c \varphi = -2 \sum_i |z_i|^2 d \theta_i + \sum_i ( O(|z|^2) dz_i + O(|z|^2) d \bar{z}_i ). \]
However, along a straight line path from $0$ to $w$, $\theta_i$ is constant, hence the leading term of $d^c \varphi$ vanishes
in the integral. For the remainder term, we have $| \int_{\gamma} d z_i | = O( |w|)$, hence proving the claim. 
 
Next, we consider a general path as in the statement of the Lemma.
For each $\epsilon$, we may consider the straight-line path $\beta: [0,\epsilon] \to U$ from $0$ to $\gamma(\epsilon)$ . From the previous claim, we know $\int_{\beta(\epsilon)} d^c \varphi = O(\epsilon^3)$. Let 
\[ \Sigma_\epsilon: [0,\epsilon] \times [0,1] \to U, (t, u) \mapsto u \gamma(t) + (1-u) \beta(t). \]
Then, we may verify that
\[ \int_{\gamma([0,\epsilon])} d^c \varphi < C \left|\int_{\Sigma_\epsilon} \omega \right| + O(\epsilon^3)<  O(\epsilon^3). \]
where the estimate of $\int_{\Sigma} \omega = 2\sum_i \int_{\Sigma} dx_i \wedge d y_i$ can be done by noting for any smooth function $f$, 
\[ \int_0^\epsilon f(x) dx - \epsilon \half(f(0) + f(\epsilon)) = O(\epsilon^3). \] 
\epf

Using  above two lemma, we have
\[ \h \theta_{w}(T) = \h \theta_{0}(T) = -\int_{\gamma} \half d^c \varphi + \int_{g^T(\gamma)} \half d^c \varphi = O(|w|^3) + O(|g^T w|^3) = O(|w|^3). \]
This finishes the proof for Proposition \ref{phase-quad}.
\epf

\subsection{\label{BFREDUCTION} Reduction to osculating BF model}
We continue the calculation of the contribution to the stationary phase integral for period $T$ orbit. 
The reduced phase function $\Psi_T'(t,w)$ has the following expansion: 
\bea 
\Psi_T'(t,w) &=& - i T E - i \h \theta_{0}(T) + i t \Re (\alpha \bar w) -  |w+i \alpha t |^2/2 - |Dg^{T} w|^2/2 +O_3. \\
&=& - i T E - i \h \theta_{0}(T) + i w \bar{\alpha} t -  |w|^2/2 - |\alpha t |^2/2 - |Dg^{T} w|^2/2 +O_3. 
\eea
We may write the critical value as
\[ \Psi_T'(0,0) = \Psi_T|_{crit} = - i T E - i \h \theta_{0}(T) = -i \theta_0^h(T) \]
using holonomy phase of the horizontal flow \eqref{hol-hor}. 

The leading term of the stationary integral can be obtained by the following model result on BF space. 
\bpp
 Let $H = \alpha \bar z + z \bar \alpha$. Let $\acal: \C^m \to \C^m$ be a symplectic linear map, $\acal w = P w + Q \bar w$. Suppose $\xi_H$ is invariant under $\acal$. Then \\
(1)
\bea
&&(\det P^*)^{1/2} \kk^{2m}\int_{\C^m} e^{k( i t w \wb{\alpha} -  |w|^2/2 - t^2|\alpha|^2/2 - |\acal w|^2/2) } d \Vol_{\C^m}(w) \\&=&  \h \kcal_{k,\acal}((0;0), \h g^t(0;0)) \\&=&  (\det P)^{-1/2} \kk^{m} e^{-kt^2(   |\alpha|^2-\bar \alpha P^{-1}Q \bar \alpha )/2}
\eea
where the metaplectic representation kernel $\h \kcal_{k,\acal}(\h z, \h w)$ is defined in \eqref{hatK}.

(2)  \[ \int_{\R} \h \kcal_{k,\acal}((0;0), \h g^t(0;0)) dt = \kk^{m-1/2} (\det P)^{-1/2} (\bar \alpha P^{-1} \alpha) .\]
\epp
\bpf
(1)
 We note that
\[ \kk^{m} e^{k( - |\acal w|^2/2) } = \h \Pi_k(0, (\acal w;0)), \] 
and
\[   \kk^{m} e^{k( i t w \wb{\alpha} -  |w|^2/2 - t^2|\alpha|^2/2)} = \h \Pi_k(\h g^{-t}(w;0), 0) =  \h \Pi_k((w;0),\h g^{t}(0; 0)). \]
Hence by  Proposition \ref{toep-met}, we have
\bea 
&& (\det P^*)^{1/2} \kk^{2m}\int_{\C^m} e^{k( i t w \wb{\alpha} -  |w|^2/2 - t^2|\alpha|^2/2 - |\acal w|^2/2) } d \Vol_{\C^m}(w)  \\&=&  (\det P^*)^{1/2} \int_{\C^m} \h \Pi_k(0, (\acal w;0))  \h \Pi_k((w;0),\h g^{t}(0; 0)) d w \\&=&  \h \kcal_{k,\acal}((0;0), \h g^t(0;0)). 
\eea
And the last line follows by $\h g^t(0;0) = (-i \alpha t; 0)$ and definition for $\h \kcal_{k,\acal}$. 

(2)
Next, we  use the fact that $\xi_H$ is preserved by $\acal$, i.e. 
\be \bma -i \alpha \\ i \bar \alpha \ema = \bma P & Q \\ \bar Q & \bar P \ema \bma -i \alpha \\ i \bar \alpha \ema  \label{alpha-inv}\ee
Thus
\be \alpha = P \alpha - Q \bar \alpha  \label{alpha-inv-2} \ee
hence
\[ |\alpha|^2 - \bar \alpha P^{-1} Q \bar \alpha = |\alpha|^2 - \bar \alpha P^{-1}  ( P\alpha - \alpha) = \bar \alpha P^{-1} \alpha \]
Then, we have
\[  \kk^{m}  (\det P)^{-1/2} \int_\R  e^{ -k  \half t^2 ( \bar \alpha P^{-1} \alpha)} dt  =  \kk^{m-1/2}  (\bar \alpha P^{-1} \alpha)^{-1/2}  (\det P)^{-1/2} \]
\epf

Combining all the steps before, we have proven the following proposition. 
\bpp
Let $z \in M$ be a periodic point for the flow $\xi_H$ and $H(z)=E$, then 
\[  \Pi^E_{k,f}(z,z) = \sum_{n \in \Z} \h f(  n T_z) e^{- i k n\theta_z^h}  \gcal_n  \kk^{m-1/2}  (1+O(k^{-1})) \]
where if $D g^{nT_z}|_z$ in K-coordinate at $z$ can be written as $\bma P_n & Q_n \\
\bar Q_n & \bar P_n \ema$, then
\[ \gcal_{n} = (\det P_n)^{-1/2} (\bar \alpha P_n^{-1} \alpha)^{-1/2}.
\]
\epp

\section{Proof of Proposition \ref{HYPLEM}}

The issue at hand is the regularity of the measures $\mu_k^{z,1,E}$ defined
on test functions $f \in \scal(\R)$ with $\hat{f} \in C_0^{\infty}(\R)$ in
Theorem \ref{PikfTH}.  It is only an interesting question when $z \in \pcal_E$.
In this case, 
$$ \int_{\R} f \mu_{k}^{z,1,E} =  \kk^{m-1/2}  \sum_{n \in \Z} \hat{f}(n T_z)
\;\gcal_n(z) e^{-i n k \theta_{z}^h} + O(k^{m-3/2}). $$
Unravelling the Fourier transform gives that, in the sense of distributions,
$$d \mu_{k}^{z,1,E}(x) = \kk^{m-1/2}  \sum_{n \in \Z} e^{i n T_zx}
\;\gcal_n(z) e^{-i k n\theta_z^h} dx+ O(k^{m-3/2}). $$

The Proposition asserts first that this series converges absolutely and uniformly when the orbit through $z$ is real  hyperbolic. To prove this
we need to consider the behavior of the matrix element $\bar{\alpha} P_n^{-1} \alpha$ and the detereminant $\det P_n$ as $n \to \infty$, where as in \eqref{PDEF} 
\begin{equation} \label{PnDEF} P_n : =  P_J S^n P_J: T_z^{(1,0)} M \to T_z^{(1,0)} M. \end{equation}
We first  develop the symplectic linear algebra introduced in  Section \ref{SLA}.

\subsection{\label{MEPDSS}Matrix elements and determinants of positive definite symplectic matrices}

We are interested in $P_J S P_J$ with $P_J = \half(I - i J) $. We also
use the notation $\langle \alpha, \beta\rangle = \bar{\beta}^t \cdot \alpha$ for the sesquilinear inner product. 

First we prove

\begin{prop}\label{PROPME}  If $S$ is positive definite symmetric symplectic, with invariant vector $\xi$ and $\alpha = P_J \xi$,  
	and if the spectrum of $S$ is $\{e^{\lambda_j}, e^{-\lambda_j} \}_{j=1}^n$ with 
	$\lambda_j \geq 0$  then $$\left\{ \begin{array}{ll} (i) & [P_J S P_J]^{-1} \alpha = \alpha, \\ & \\
	(ii)& \det P_J S P_J|_{T^{1,0}_0\R^{2n}} = \prod_{j=1}^n [\cosh \lambda_j]. \end{array} \right. $$\end{prop}
\begin{proof}
	
	The proof is through a series of Lemmas:
	
	\begin{lem} If $S$ is  positive definite symplectic, then 
		$$P_J S P_J =\half P_J (S + S^{-1}) = \half (S + S^{-1}) P_J$$

	\end{lem}
	
	\begin{proof}
		$$\begin{array}{lll} P_J S P_J & = & \frac{1}{4} (I - i J)  S (I - i J) = \frac{1}{4}[ S - i JS - i SJ - J SJ ]  \\&&\\& = & 
		 \frac{1}{4} [S + S^{-1}] - \frac{i}{4}[J [S + S^{-1}] = \frac{1}{4}\left( (S + S^{-1}) - i J (S + S^{-1})\right)
		= \frac{1}{2} P_J (S + S^{-1}).
		\end{array}$$
		since $JSJ = - S^{-1}   $ if $S$ is symmetric. 
		Also,
		$$J (S + S^{-1} ) = JS + SJ = (S^{-1} + S)J$$
		so that $P_J (S + S^{-1}) = (S + S^{-1}) P_J. $
	\end{proof}

	\begin{lem} \label{EIGLEM} Let $S$ be  positive definite symmetric symplectic and  $e_j$ be eigenvectors of $S$ for eigenvalues $\lambda_1, \dots, \lambda_n$.
		Consider the basis $P_J e_k$ of $H^{1,0}_J$. Then  $$[P_J S P_J] P_J e_k =  \cosh(\lambda_j) P_J e_k, $$
		and $[P_J S P_J]^{-1} = P_J [S + S^{-1}]^{-1} P_J. $ 
	\end{lem}
	
	\begin{proof}
		
		Follows from the previous Lemma  and the fact that $(S + S^{-1})$ commutes with $P_J$:
		\[ [P_J S P_J] P_J e_k = \half P_J(S+S^{-1}) e_k = \half (e^{\lambda_j}+ e^{-\lambda_j}) P_J e_k = \cosh(\lambda_j)P_J e_k. \]
		 \end{proof}

	Statement (i) of the Proposition follows from the fact that
	$$[P_J S P_J] \alpha = \half (1+1)\alpha
	= \alpha. $$
	
	Statement (ii) follows from the fact that the eigenvalues of $P_J S P_J$ are 
	$\cosh \lambda_j$ by Lemma \ref{EIGLEM}.

\end{proof}

\subsection{Strong Hyperbolicity Hypothesis}
Let $z$ be a periodic point of the Hamiltonian flow $g^t$. 
Under this hypothesis, we have the following result. 
\bpp \label{Gconv}
If $\dim_\C M = m > 1$, and $z$ be a periodic point with primitive period $T$, satisfying the strong hyperbolic hypothesis. Then 
\[ \sum_{n \in \Z} |\gcal_n(z)| < \infty.\]
\epp
\bpf
Let the spectrum of $S:=D g^T$ be $\{e^{\lambda_j}, e^{-\lambda_j} \}_{j=1}^m$, with $\lambda_1=0$ and $\lambda_j > 0$ for $j=2,\cdots,n$. Then, recall that 
\[  \gcal_n(z) = [ \det (P_J S^n P_J) \la (P_J S^n P_J)^{-1} \alpha, \alpha \ra ]^{-1/2} . \]
Then, from previous section, we have $ \det (P_J S^n P_J) = \prod_{j=1}^n \cosh(n \lambda_j)$, and $\la (P_J S^n P_J) \alpha, \alpha \ra = \la \alpha, \alpha \ra$ independent of $n$. Since $\lambda_j > 0$ for $j=2, \cdots, m$, hence 
\[ |\gcal_n| = |\det (P_J S^n P_J) \la \alpha, \alpha \ra |^{-1/2} < C e^{ -|n|\sum_j \lambda_j} \]
for some positive constant $C$. Thus the sum $\sum_{n \in \Z} |\gcal_n(z)|$ converges exponentially fast. 
\epf

\subsection{Proof of Proposition \ref{HYPLEM}}

By Proposition \ref{Gconv}, the  family of measures 

$$d\nu_T(\lambda) := \sum_{|n| \leq T}  \rho_T( n T(z)) e^{-i \lambda n T_z} e^{-ikn \theta^h_z(T_z)} 
\;\gcal_n(z) d \lambda, \;\;\; (T \in \R_+)$$
converges in the weak* sense of distributions on the space $\scal(\R)$
of Schwartz functions to the limit distribution, 
\begin{equation} \label{nudef} d\nu(\lambda) := \sum_{n \in \Z} e^{-i \lambda n T_z} e^{-ikn \theta^h_z(T_z)} 
\;\gcal_n(z) d \lambda, \end{equation}
since the coefficients $\gcal_n(z) $ are bounded in $n$ and  by dominated convergence, $$\int_{\R} f(\lambda) d\nu_T(\lambda) =
\sum_{|n| \leq T} \rho_T(n T(z))  \hat{f} (n T(z))
\;\gcal_n(z)  \to 
\sum_{n \in \Z}   \hat{f} (n T(z))
\;\gcal_n(z), $$
where the sum on the right side converges absolutely.

\section{\label{TAUBERPROOF} Proof of Theorem \ref{2TERM}}

In this section we apply Theorem \ref{PikfTH}  and a Tauberian theorem to prove Theorem \ref{2TERM}.  We are concerned with the Weyl sums, \begin{equation}\label{DOSsharp} 
\begin{array}{l} 
\Pi_{k, [E_1, E_2]}(z) = 
\int_{E_1}^{E_2} d\mu_k^{z,1,E} = \sum_{j : k (\mu_{kj} - H(z)) \in [E_1, E_2]}
\Pi_{k, j}(z). \end{array} 
\end{equation}
The basic idea is to  convolve ${\bf 1}_{[E_1, E_2]}$ with a well-chosen  Schwartz test function depending on $(h, T)$, apply Theorem \ref{PikfTH}  and 
then estimate the remainder.

We consider both families of  measures of \eqref{mukzdef}, $\mu_k^z$ and $\mu_k^{z, 1, E}$. The main difference is the range of eigenvalues involved. 
The measures $\mu_k^z$ have a fixed compact support, the range $H(M) = [H_{\min}, H_{\max}]$ of $H$, and the mean level spacing between the $k^m$ point masses $\mu_{kj}$  is $k^{-m}$.  The measures $\mu_k^{z, 1, E}$ are scaled versions,
$$\mu_k^{z,1, E}[-M,M] = \sum_{j: |\mu_{jk} - E| < \frac{M}{k}} \Pi_{kj}(z),  $$
and the mean level spacing between the point masses is $k^{-m + 1}.$ Of course, 
\begin{equation} \label{M2} \sum_{j: |\mu_{jk} - E| < \frac{M}{k}} \Pi_{kj}(z) = \mu_k^{z,1,E}[-M, M] =\mu_k^z[\frac{-M}{k}, \frac{M}{k}], \end{equation}

As a preliminary, we quote a result from \cite[Theorem 3]{ZZ17}:

\begin{theo} \label{ELLSMOOTH} 
	Let $E$ be a regular value of $H$ and $z \in H^{-1}(E)$. If $\epsilon$ is small enough, such that the Hamiltonian flow trajectory starting at $z$ does not return to $z$ for time $|t| < 2\pi \epsilon$, then for any Schwarz function 
	$f \in \scal(\R)$ with $\hat{f}$ supported in $(-\epsilon, \epsilon)$ and $\h{f}(0) = \int f(x) dx = 1$, and for any $\alpha \in \R$ we have 
	\[
	\int_\R f(x) d \mu^{z, 1, \alpha}_k(x) = \kk^{m-1/2} e^{- \frac{\alpha^2}{\|\xi_H(z)\|^2}} \frac{\sqrt{2}}{2\pi \|\xi_H(z)\|}(1+ O(k^{-1/2})).
	\]
\end{theo}

There is a further integrated version of the Weyl law with remainder, \begin{equation} \label{IWL} \#\{j: |\mu_{kj} - E| \leq \frac{M}{k}\} = \frac{2 M}{(2 \pi)^n} {\Vol}(h^{-1}(E)) k^{m-1} + o(k^{m-1}). \end{equation} The constraint in the sum \eqref{M2} is a `codimension one' condition localizing around $H^{-1}(E)$.  
The
extra integration in \eqref{IWL} gives an extra factor of $k^{-\half}$ in the stationary phase 
expansion. Note that $\int_M \Pi_{kj}(z) dV(z) = \rm{Mult}(\mu_{kj})$ (the multiplicity of the eigenvalue, generically equal to $1$), so the integrated
Weyl law does not deal with non-uniform weights $\Pi_{kj}(z)$. The integrated Weyl law (essentially contained in \cite{BG81}.

The remainder  estimate requires the use of a semi-classical Tauberian theorem for a sequence $\mu_k^{z, 1, E}$ of measures. Before getting started, let us note some basic facts about this sequence. First, $\mu_k^{z, 1, E}$ is not normalized to be a probabiity measure, but it is finite and could  be normalized  by dividing by its mass $\Pi_{h^k}(z) \simeq k^m + O(k^{m-1}))$. 
In the following discussion, we divide by the mass. 
Second, note that $\Pi_{h^k}(z) ^{-1}d\mu_k^{z,1,E} $ is a centered re-scaling of $\Pi_{h^k}(z) ^{-1}d\mu_k^z $ \eqref{mukzdef}. That is $D_{k} \tau_E d\mu_k^{z, 1 E} = d\mu_k^z$
where the dilation operator is defined by  $D_k \nu(I) = \nu(k I)$ for any interval $I$ and measure $\nu$. Also $\tau_E f(x) = f(x -E)$. Now, $\mu_k^z$
is supported in $H(M)$ (the range of $H: M \to \R$), hence $\mu_k^{z, 1, E}$ is supported in $k (H-E) (M)$. In \cite{ZZ17} we studied $\Pi_{h^k}(z) ^{-1} \mu_k^{z, \half, E} : = D_{\sqrt{k}} \Pi_{h^k}(z) ^{-1} \mu_k^{z, 1, E}$, whose support is $\sqrt{k} (H -E) (M)$  and proved that it tends to a Gaussian. In particular, its Fourier transform is continuous at $0$, and by Levy's continuity theorem (or by  direct analysis), the sequence  $\Pi_{h^k}(z) ^{-1} \mu_k^{z, 1/2, E}$  is tight. By comparison, $\Pi_{h^k}(z) ^{-1} \mu_k^{z, 1, E}$ is not tight, and indeed the   $\Pi_{h^k}(z) ^{-1} \mu_k^{z,1,E} ([a,b]) \simeq k^{-\half}$, so that the mass is spreading out to infinity and it does not
weak* converge on $C_b(\R)$.

Theorem \ref{2TERM} not only gives the leading order term but also the order of the remainder.  As is well-known from work of Duistermaat-Guillemin, Ivrii, Safarov and others, obtaining a sharp  remainder term requires the use of something similar to Fourier transform methods and in particular Fourier Tauberian theorems. 
As mentioned before, Theorem \ref{2TERM} is analogous to Safarov's non-classical  pointwise Weyl asymptotics for the spectral function of a Laplace operator $\Delta$, or more precisely, asymptotics on intervals $[\lambda, \lambda + 1]$ for $\sqrt{-\Delta}$. The $Q$-notation is adopted from \cite{Sa,SV}. Since we are working on phase space, $Q$ involves closed orbits rather than loops in configuration space. However, we need to use a semi-classical Tauberian theorem rather than the homogeneous Tauberian theorem of \cite{SV}, i.e. we are considering a sequence of measures $\mu_k^{z,1,E}$ on a fixed interval rather than a fixed measure on  expanding  intervals $[0, \lambda]$.

Semi-classical Tauberian theorems have been known for a long time. It is a classical fact that to obtain sharp remainder estimates, one must make use of the
Fourier transform of the measures on long time intervals $[-T, T]$. A Tauberian theorem of the needed type is proved in  \cite{PR85}, adapting the statement of Safarov's non-classical Weyl asymptotics to a semi-classical problem. 
This theorem does not quite apply to our setting for various reasons: (i) It assumes the sequence of measures have fixed compact support; (ii) it assumes the `weights' or masses of the point masses are uniform.  On the contrary, the  `weights' $\Pi_{k,j}(z)$ of $\mu_k^{z, 1, E} $  are highly non-uniform in a way that is inconsistent with the hypotheses of the Tauberian Theorem of \cite{PR85}.  Consider the graph of the weights $\Pi_{k,j}(z)$ as a function of $\mu_{kj}$, i.e. the coefficients of the
point masses of $\mu_k^z$ \eqref{mukzdef}. On average the weights are of order $1$ since there
are $k^m$ terms and the total sum is $\Pi_k(z) \simeq \rm{Vol}(M, \omega) k^m$. But the weights are highly non-uniform: 

\begin{enumerate}
	
	\item they  peak when $\mu_{kj} \simeq H(z)$; indeed, it
	is shown inf \cite[Theorem 1]{ZZ17} that $\mu_{k}^z$ tends weakly to  $\delta_{H(z)}$. \bigskip
	
	\item By \cite[Theorem 2]{ZZ17},  $\sum_{j: |\mu_{kj} - H(z)| < M  k^{-\half}} \Pi_{k, j}(z)  \sim M k^m$ while the number of terms
	is of order $k^{m-\half}$. Thus, on average, $\Pi_{k,j}(z)$ is of size $k^{\half}$ in this eigenvalue range. \bigskip
	
	\item Further, $\Pi_{k, j}(z) \lesssim k^{-C}$ when
	$|H(z) - \mu_{kj} | \geq C k^{-\half} \log k$. Hence, the weights decay rapidly when $\mu_{kj}$ lies outside of the range $|H(z) - \mu_{kj} | \leq C k^{-\half} \log k$. Consequently, 
	the sequence of dilated measures  $\mu_k^{z,1,E}$ 
	concentrates in the sets $[-k^{\half} \log k, k^{\half} \log k]$.
	
\end{enumerate}

Since we need to modify the Tauberian  Theorem of \cite{PR85}  to accomodate the strong  peaking of the weights around $H(z)$, we go through the modified proof in detail.

\subsection{\label{MCSECT} Mollifiers and convolution}
We use the following notation: 
Let $\rho_1 \in C_0^{\infty}(-1,1) $ satisfy $\rho_1(t) = 1 $ on $[-\half, \half], \rho_1(-t) = \rho_1(t)$. We may also
assume $\fcal {\rho_1}(\tau) \geq 0$ and $ \fcal {\rho_1}(\tau) \geq \delta_0 > 0$ for $|\tau| \leq \epsilon_0$, where
 $\fcal$ and $\fcal^{-1}$ denote the standard Fourier transform and its inverse,  
\[ \h f(x) := (\fcal f)(x) = (2\pi)^{-1} \int f(t) e^{- i t x}  dt, \quad \check f(x) = (\fcal^{-1} f)(x) = \int f(t) e^{it x} dx\]
Then set,
\begin{equation} \label{thetaDEF} \rho_T(\tau) = \rho_1(\frac{\tau}{T}), \;\;\;\theta_{T}(x) := \h \rho_T (x)  = T \h  \rho_1(x T). \end{equation}
In particular, $\int \theta_T(x) dx = 1$ and $\theta_T(x) > T \delta_0$ for $|x| < \epsilon_0/T$. 
Let 
\begin{equation}  \sigma_k^{z, 1, E}(x) = \mu_k^{z,1, E} (-\infty, x]. \end{equation}

\subsection{Tauberian theorem for $\mu_k^{z,1,E}$}

In this section we determine the asymptotics of 
\begin{equation} \label{DIFFSIGMAS} \sigma_k^{z,1,E}(E_2) - \sigma_k^{z,1, E}(E_1) =  \int_{E_1}^{E_2} d\mu_k^{z,1,E}(x) = 
\sum_{j: \frac{E_1}{k} \leq \mu_{jk} - E  \leq \frac{E_2}{k}} \Pi_{k,j}(z). \end{equation}
We recall that the mean level spacings of $k (\mu_{k,j} - E)$ is $k^{-m + 1}$
so that the number of terms in the sum is of order $k^{m-1}$.
The plan is to mollify the measures by convolution with $\theta_{T}$  \eqref{thetaDEF}, 
so that it suffices to determine the asymptotics of
\begin{equation} \label{MAINLINE} \begin{array}{l} \sigma_k^{z,1,E} * \theta_T (E_2) - \sigma_k^{z,1, E} * \theta_T (E_1) \\ \\
+ \left(\sigma_k^{z,1,E}(E_2) - \sigma_k^{z, 1, E}(E_1) \right)-  \left(\sigma_k^{z,1,E}* \theta_T (E_2)- \sigma_k^{z, 1, E}  * \theta_T (E_1) \right)
\end{array} \end{equation}
Since
$$ \sigma_k^{z,1,E} * \theta_T (E_2) - \sigma_k^{z,1, E} * \theta_T (E_1)  = \int_{E_1}^{E_2} \theta_{h, T} * d\mu_k^{z, 1,E}(\lambda),  $$ we have
\begin{equation}\label{NEEDT}
\left(\sigma_k^{z,1,E}(E_2) - \sigma_k^{z, 1, E}(E_1) \right)-  \left(\sigma_k^{z,1,E}* \theta_T (E_2)- \sigma_k^{z, 1, E}  * \theta_T (E_1) \right)= \int_{E_1}^{E_2} (\theta_T * d\mu_k^{z, 1, E} - d\mu_k^{z, 1, E} ). \end{equation}

First we consider the top terms of \eqref{MAINLINE}. 
\begin{prop} \label{TOPTERM}Assume that $H(z) = E, z \in \pcal_E$. Then
\be \frac{d}{dx} (\sigma_k^{z, 1, E} * \theta_T) (x) = \kk^{m-1/2} \sum_{n \in \Z} \rho_T(n T_z) e^{-i x n T_z} e^{-ikn \theta^h_z(T_z)}  \gcal_n(z) + O_T(k^{m-3/2})  \label{diffsigma} \ee
and
$$ \sigma_k^{z,1,E} * \theta_T (E_2) - \sigma_k^{z,1, E} * \theta_T (E_1)  = k^{m-\half} \int_{E_1}^{E_2}  \sum_{|n T_z| \leq T}  \rho_T( n T_z) e^{-i \lambda n T_z} e^{-ikn \theta^h_z(T_z)} 
	\;\gcal_n(z) d \lambda + O(k^{m-1}),\;  . $$
\end{prop} 

\begin{proof}

\bea
 \frac{d}{dx} (\sigma_k^{z, 1, E} * \theta_T) (x) &=& \int \theta_T(x-y) d \mu_k^{z,1, E} (y) \\
&=& \int_\R \int_\R \rho_T(-t) e^{-i t (x-y)} \sum_j \delta_{k (\mu_{k,j} - E)}(y) \Pi_{k,j}(z) dy d t \\
&=& \int_\R \rho_T(t) e^{-itx}  \sum_j e^{it k (\mu_{k,j} - E) } \Pi_{k,j}(z) dt \\
&=& \int_\R \rho_T(t) e^{-itx-itkE} U_k(t, z,z) dt \\
&=& \kk^{m-1/2} \sum_{n \in \Z} \rho_T(n T_z) e^{-i x n T_z} e^{-ikn \theta^h_z(T_z)}  \gcal_n(z)  (1+O(k^{-1})). 
\eea
where the last line follows from Theorem \ref{PikfTH} to $f(y)= \theta_{T}(x -y)$. 

\end{proof}

\bc
Under the strong hyperbolicity hypothesis (Definition \ref{hyperhypo}), there exists constants $\gamma_0(z), C_1(T,z)$, such that
\[ \frac{d}{dx} (\sigma_k^{z, 1, E} * \theta_T) (x) \leq \kk^{m-1/2} \gamma_0(z) +  C_1(T, z) k^{m-3/2}.\]
\ec
\bpf
We start from \eqref{diffsigma}, and let $T \to \infty$.  By Proposition \ref{Gconv}, the sum in \eqref{diffsigma} with $\rho_T$ replaced by $1$ converges absolutely. 
\epf
 
We now employ a semi-classical Fourier Tauberian theorem to  estimate \eqref{NEEDT}.  In fact, since we already semi-classically scaled $d\mu_k^z$ by $k$, we do not need to scale again. We only refer to the Tauberian as semi-classical because it applies to a sequence $\mu_k^{z, 1, E}$ of measures on a fixed interval rather than to a fixed measure on a dilated family of intervals as in the homogeneous Tauberian theorem.

The Tauberian theorem  states:

\begin{prop}\label{TLEM} There  exist constant $\gamma(z ), C(T,z) $ such that, for any $T >0$,
	$$
	\int_{E_1}^{E_2} (\theta_T * d\mu_k^{z, 1, E} - d\mu_k^{z, 1, E}) \leq \frac{\gamma(z)}{T} k^{m - \half} + C(T, z) k^{m-3/2}. $$ 
\end{prop}

Together with Proposition \ref{TOPTERM} this gives
\begin{cor} \label{TAUBERCOR} For any $T > 0$, there  exist $\gamma_0(z, \tau), \gamma, C_1(T, z, \tau) > 0$ so that  $$ \sigma_k^{z,1,E} (E_2) - \sigma_k^{z,1, E} (E_1)  = \kk^{m-\half} \int_{E_1}^{E_2}  \sum_{|n T_z| \leq T}  \rho_T( n T_z) e^{-i \lambda n T_z} e^{-ikn \theta^h_z(T_z)} 
	\;\gcal_n(z) d \lambda + \frac{1}{T} O(k^{m-\half})+\;  O_T (k^{m-3/2}). $$
	
\end{cor}

\subsection{Proof of Proposition \ref{TLEM}}
As mentioned above, the hypotheses 
of \cite[Theorem 3.1]{PR85} do not hold in our setting. Hence we must extract from \cite[Theorem 3.1]{PR85} the key elements that pertain to our setting.

We have,
$$ \begin{array}{lll}  \int_{E_1}^{E_2} (\theta_T * d\mu_k^{z, 1, E} - d\mu_k^{z, 1, E}) & = & 
\int_{\R} \left(\mu_k( [E_1, E_2] - \tau)  - \mu_k[E_1, E_2]\right) \theta_{T}(\tau) d\tau\\&&\\
& = &T \int_{\R}  \left( \mu_k([E_1, E_2] - \tau) - \mu_k[E_1, E_2]) \right)  \h{\rho}_1( \tau T) d \tau \\&&\\
& = &T  \int_{|\tau| \leq \frac{1}{T} }    \left( \mu_k([E_1, E_2] - \tau)  - \mu_k[E_1, E_2]) \right) \h{\rho}_1( \tau T) d \tau \\&&\\
&  + & T \int_{|\tau| > \frac{1}{T}}    \left( \mu_k( [E_1, E_2]- \tau)  - \mu_k[E_1, E_2]) \right)   \hat{\rho}_1( \tau T) d \tau \\&&\\& =: & I_1  + I_2. \end{array}$$

Evidently, the key objects to estimate are the increments

\begin{equation} \label{INCREMENTS}  \mu_k( [E_1, E_2] - \tau)  - \mu_k([E_1, E_2])  \end{equation}
The key point is to prove the analogue of \cite[Proposition 3.2]{PR85}: 

\begin{prop}\label{TLEM2} There  exist constants $\gamma_1(z)$ and $C_1(T,z)$  such that, for any $T > 0$,  
	$$\left|\left( \mu_k([E_1, E_2] - \tau)  - \mu_k[E_1, E_2]) \right) \right| \leq  \gamma_1(z)   (\frac{1}{T} + |\tau|)  k^{m - \half} + C_1(T,z) O(k^{m-3/2}) $$ 
	
\end{prop}

We now show
that Proposition \ref{TLEM2} implies Proposition \ref{TLEM}.

\begin{proof} First,  observe that Proposition \ref{TLEM2} implies,
	\begin{equation} \label{Izbd} I_1  \leq \sup_{|\tau| \leq \frac{1}{T}} 
	\left|   \mu_k( [E_1, E_2] - \tau)  - \mu_k([E_1, E_2])\right|, \end{equation}
	and  Proposition \ref{TLEM2} immediately implies the desired bound of Proposition \ref{TLEM} for $|\tau| \leq \frac{1}{T}$.   For $I_2$ one uses that   $\hat{\rho}_1 \in \scal(\R)$.
	Since
	$T \int_{|\tau| \geq \frac{1}{T} } \hat{\rho}_1(\tau T) d \tau \leq 1, $ Proposition \ref{TLEM2} implies,
	\begin{equation} \label{IIzbd} I_2  \lesssim
	k^{m-\half}\gamma_1(z)   T \int (\frac{1}{T} + |\tau|) 
	\hat{\rho}_1( T \tau) d \tau +  C_1(T,z) O(k^{m-3/2}) T \int_{|\tau| > \frac{1}{T}}    \hat{\rho}_1( \tau T) d \tau  \end{equation}
	If one changes variables to $r = T \tau$ one also gets the estimate of Proposition \ref{TLEM}.\end{proof}

We now prove Proposition \ref{TLEM2}.

\begin{proof} We need to estimate $\left( \mu_k [E_1, E_2] - \tau)  - \mu_k[E_1, E_2]) \right) $. The estimate depends both on the position 
	of $[E_1, E_2]$  relative to the center of mass at $0$ and on the position
	of $\tau$. We recall the the total mass of $\mu_k = \mu_k^{z, 1, E}$ on
	the complement of $[- \sqrt{k} \log k, \sqrt{k} \log k]$ is rapidly decaying in $k$. Hence we may assume that at least one of the following occurs:

	\begin{itemize}
		
		\item $[E_1, E_2] \cap [- \sqrt{k} \log k, \sqrt{k} \log k] \not= \emptyset,$ i.e. 
		$E_1 \geq - \sqrt{k} \log k, E_2 \leq \sqrt{k} \log k$.
		\bigskip
		
		\item  $ [E_1, E_2] - \tau \cap [- \sqrt{k} \log k, \sqrt{k} \log k] \not= \emptyset, $ i.e. $E_1- \tau  - \sqrt{k} \log k,  E_2 - \tau \leq \sqrt{k} \log k$.
		
	\end{itemize}
	
	The proof is broken up into 3 cases: (1) $|\tau | \leq \frac{\epsilon_0}{T}$,
	\;  (2) $\tau = \frac{\ell}{T} \epsilon_0,$\; (3)  $ \frac{\ell}{T} \epsilon_0 \leq \tau \leq  \frac{\ell+1}{T} \epsilon_0 $, for some $ \ell \in \Z$. 
	 
	\begin{enumerate}
		
		\item Assume $|\tau| \leq \frac{\epsilon_0}{T}$. Assume $\tau > 0$ since the
		case $\tau < 0$ is similar.  Write 
		$$ \begin{array}{lll} \mu_k([E_1, E_2] -\tau)  - \mu_k[E_1, E_2])  & = & \int_{\R} [{\bf 1}_{[E_1 - \tau, E_2- \tau]} - {\bf 1}_{[E_1, E_2]}] (x)d \mu_k(x)
		. \end{array} $$
		For $T$ sufficiently large so that $ \tau \ll E_2 - E_1$, $$ [{\bf 1}_{[E_1 - \tau, E_2- \tau]} - {\bf 1}_{[E_1, E_2]}] (x)= {\bf 1}_{[E_1 - \tau, E_1]} - {\bf 1}_{[E_2 - \tau, E_2]}. $$
		We do not expect cancellation between the terms for arbitrary $E_1, E_2, \tau$ and therefore must show that  each term satisfies the desired estimate.
		Since they are similar we only consider the $[E_1 - \tau, E_1]$ interval. Since for $|\tau| < \epsilon_0/T$, we have $\theta_T(\tau) > T \delta_0$, thus
\bea
\mu_k([E_1 - \tau, E_1]) &\leq& \frac{1}{T \delta_0} \int_\R \theta_T(E_1-x) d \mu_k(x) \\
&\sim&  \frac{1}{T \delta_0} \frac{d}{dx} (\sigma_k^{z, 1, E} * \theta_T) (E_1)\\
&<& \frac{\gamma_0(z)}{T \delta_0} k^{m-1/2}
\eea

%
%
%
		It follows that
		$$  \left| \mu_k([E_1, E_2] -\tau)  - \mu_k[E_1, E_2])  \right| \leq \frac{2\gamma_0(z)}{T \delta_0} k^{m-1/2}. $$
	
		\bigskip

		\item Assume $\tau = \ell \frac{\epsilon_0}{T} , \ell \in \Z.$ With no loss of generality, we  may assume  
		$\ell \geq 1.$ Write
		$$ \mu_k([E_1, E_2])  - \mu_k([E_1, E_2]- \frac{\ell}{T} \epsilon_0)  
		= \sum_{j = 1}^{\ell} \mu_k([E_1, E_2] -  \frac{j-1}{T} \epsilon_0 )- \mu_k([E_1, E_2] - \frac{j}{T} \epsilon_0 )   $$
		and apply the estimate of (1) to upper bound the sum by 
		$$\frac{2 \ell \gamma_0(z)}{T \delta_0} k^{m-1/2} = \frac{2 \gamma_0}{\epsilon_0 \delta_0} \tau k^{m-1/2}$$
		
		\bigskip 
		
		\item Assume $ \frac{\ell}{T} \epsilon_0 \leq \tau \leq  \frac{\ell+1}{T} \epsilon_0$ and
		$|\tau h| \leq \epsilon_1$ with $\ell \in \Z$. Write
		$$\begin{array}{lll} \mu_k([E_1, E_2] + \tau ) - \mu_k([E_1,E_2]) & = &  \mu_k([E_1,E_2] + \tau ) -  \mu_k([E_1,E_2] + \frac{\ell}{T} \epsilon_0 ) \\&&\\&&+ \mu_k([E_1,E_2] + \frac{\ell}{T} \epsilon_0 )
		- \mu_k([E_1,E_2]).\end{array} $$
		Apply (1) and (2) , it follows that
		$$|\mu_k([E_1, E_2] + \tau ) - \mu_k([E_1,E_2])| \leq \frac{2 \gamma_0(z)}{\delta_0}(\frac{\tau}{\epsilon_0} + \frac{1}{T}) \gamma_0(\sigma, \lambda) k^{m-\half}. $$ \bigskip

	\end{enumerate}

\end{proof}

\section{\label{BPUSECT} Comparison with BPU}

In this section we compare our formula for the leading coefficient in Theorem \ref{PikfTH} with that in \cite{BPU98}. To do so, we need to introduce the
notation and terminology of that article. 

 Let $\phi_{\tau}^h$ be the horizontal lift of the Hamiltonian flow to $X_h$ (denoted $P$ in \cite{BPU98}). At each point $p \in P$,  define $T_p^h P$ to be the horizontal subspace and $\Lambda_p$ to be the positive definite Lagrangian subspace of $T_p^h P \otimes \C$ (i.e. the type $(1,0)$ subspace).  By the analysis of \cite[Page 98]{BG81} there exists a one-dimensional kernel $\wcal_p$ of this action, the line
of ground states $\wcal_p  \subset H_{\infty}(T_p^h P)$. A normalized section of the bundle $\wcal \to P$ defined by $\wcal_p$ is denoted by $e_p$. Further denote by $M_{\tau}: H_{\infty}(T_p^h P) \to H_{\infty}(T_p^h P)$ the metaplectic representation of the symplectic group of the 
horizontal space $H(T_p^h P)$. 

Let $\Xi$ denote the Hamilton vector field $\xi_H$ . It is written in  \cite{BPU98} that ``$\Xi$ acts on $H(T_p^h P)$ and hence on 
$H_{\infty}(T_p^h P)$ by via the Heisenberg representation. The action is by translations. The projection
from $H_{\infty}(T_p^h(P))$ to generalized invariant vectors under $\Xi$ is defined by

$$P_{\Xi}v : =\int_{-\infty}^{\infty} e^{it \Xi} v dt $$
the projection from $ H_{\infty}(T_p^h P) $ to the invariant vectors for the flow of $\Xi$ 
$p$ above $z$. 

Further let $Q$ be a first order pseudo-differential operator on $L^2(P)$ so that $\Pi Q \Pi  = D \Pi M_H \Pi$ and so that $[Q, \Pi] = 0$. Let $q $ be the symbol of $Q$, which generates a contact flow $\phi_t$ on $P$. Then the flow maps
$\Lambda_p \to \Lambda_{\phi_t(p)}$ and $M_{\tau} $ maps $e_p$ to 
a multiple of $e_{\phi_t(p))}$. Define $c(t)$ by $\Xi_q e_{\phi_t(p)} = i c(t)
e_{\phi_t(p)}$. 

Then the formula of \cite{BPU98} for the leading coefficient at a periodic
orbit of period $\tau$ is
\begin{equation} \label{BPUFORM} C_{\tau, 0} = \frac{1}{2 \pi^{n+1}} \langle M_{\tau}^{-1} e_{p_1}, P_{\Xi} (e_{p_1}) \rangle e^{- i \int_0^{\tau} (\sigma_{sub}(Q) + c(t)) dt} . \end{equation}

The approach of this paper is to replace $H_{\infty}(T_p^h)$ by the 
osculating Bargmann-Fock space, i.e. the Bargmann-Fock space on 
$H^{1,0}_z M$ which carries a complex structure and Hermitian metric and hence a Gaussian inner product.  In effect, the quadratic part of the scaled phase
of $U_k(t, z, z)$ replaces the symbol calculus. We do not use   $Q$ but
the related operator in our setting is $\hat{H}_k$. The $P_{\Xi}$ operator there corresponds to the $dt$ integral near a period in our approach. We now verify that our formula agrees with theirs to the extent possible.

We would like to compare the expression  \eqref{gcalndef}  with the one in 
\cite{BPU98},
\[ \la M_{T}^{-1} e_0, P_{\Xi} e_0 \ra =   \la \eta_{J,D g^{T}} \Pi_J U_{D g^{T}}^{-1} e_0,  \int_{\R} g^{BF, \tau}_* e_0  d\tau \ra= \eta_{J,D g^{T}} \int_{\R}  \la  U_{D g^{T}}^{-1} e_0,   g^{BF, \tau}_* e_0 \ra d\tau\]
where $g^\tau$ is the BF translation (Heisenberg representation) of the constant vector field $\xi_H(0)$ by time $\tau$. Here,  we dropped the projection operator $\Pi_J$, since it  is  acting on $g^{BF, \tau}_* e_0$, which is holomorphic already.

Let 
\[ v = e^{-k|z|^2/2} \]
be the (unnormalized) coherent state centered at $0$. We first review how Heisenberg group and Metaplectic group acts on it. \\
(i) Let $w \in \C^m$. Let $\beta(w)$ be translation by $w$. Then 
\[  [\beta(w) v](z) = e^{k [z \bar w - |z|^2/2 - |w|^2/2]} = e^{k[i \Im(z \bar w) - |z-w|^2/2]} \]
Indeed, it is centered at $w$, with a non-trivial phase factor $i \Im(z \bar w)$. 

(ii) Let $M = \bma P & Q \\ \bar Q & \bar P \ema \in Sp_c$, with $M^{-1} = \bma P^* & - Q^t \\ -Q^* & P^t \ema$. Then
\[ (M v)(z) = \frac{1}{(\det P)^{1/2}} e^{k [ \half z \bar Q P^{-1} z - \half |z|^2 ]}. \]
And for our purpose, we also need
\[ (M^{-1} v) (z) =  \frac{1}{(\det P^*)^{1/2}} e^{k [ - \half z  Q^* (P^*)^{-1} z - \half |z|^2]} \]

Let $\Xi = - i \alpha \pa_z + i \bar \alpha \pa_{\bar z}$, the Hamiltonian vector field for $H = \alpha \bar z + \bar \alpha z$. Then, we can write $P_\Xi v$ as
\[ (P_\Xi v)(z) = \int_\R \beta(- i \alpha t) v dt = \int_\R e^{k [i t z \bar \alpha - |z|^2/2 - |\alpha t|^2/2]} dt \]
It is possible to perform the Gaussian integral, then we get
\[ (P_\Xi v)(z)  = \sqrt{\frac{2\pi}{k |\alpha|^2}} e^{k[- |z|^2/2  - (z \bar \alpha)^2/2 |\alpha|^2)  ]}  \]
We will see, it is better not to evaluate the $dt$ integral first. 

\bpp
\[ \la M^{-1} v, P_\Xi v \ra = \kk^{-m-1/2} (\bar \alpha (P^*)^{-1} \alpha)^{-1/2}  (\det P^*)^{-1/2} \]
\epp
The power of $\kk$ does not matter, since we did not choose a normalized coherent state. The difference between $P$ and $P^*$ with previous result may be due to the difference of time $+T$ or $-T$ trajectories. Since we will sum time $\{ n T \mid n \in \Z\}$ trajectories, the difference does not matter in the end. 
\bpf
\bea
\la M^{-1} v, P_\Xi v \ra &:= & \int_{\C^m} \int_\R \frac{1}{(\det P^*)^{1/2}} e^{k [ -z  Q^* (P^*)^{-1} z - |z|^2/2 ]} \wb{e^{k [i t z \bar \alpha - |z|^2/2 - |\alpha t|^2/2]}} dt d \Vol(z) \\
&=&\int_\R \int_{\C^m} e^{k [ -i t  \bar z \alpha -z  Q^* (P^*)^{-1} z/2 - |z|^2   - |\alpha t|^2/2]} d \Vol(z) dt \\
&=&\int_\R \int_{\C^m} e^{- \half k \Psi(t, z)} d \Vol(z) dt \\
\eea
Let us do the complex Gaussian integral. The phase function is quadratic
\[ \Psi =  \bma t & z^t & \bar z^t \ema   \bma |\alpha|^2 & 0 & -i  \alpha^t \\
0 & Q^* (P^*)^{-1} & I \\ -i  \alpha & I & 0 \ema    \bma t \\ z \\ \bar z \ema \] 
We have
\[ \det   \bma |\alpha|^2 & 0 & -i  \alpha^t \\
0 & Q^* (P^*)^{-1} & I \\ -i  \alpha & I & 0 \ema  =  \det   \bma |\alpha|^2 & i  \alpha^t Q^* (P^*)^{-1}  & -i  \alpha^t \\
0 & 0 & I \\ -i  \alpha & I & 0 \ema \] 
\[= \det  \bma |\alpha|^2 - \alpha^t Q^* (P^*)^{-1} \alpha & i  \alpha^t Q^* (P^*)^{-1}  & -i  \alpha^t \\
0 & 0 & I \\0 & I & 0 \ema  = (-1)^n (|\alpha|^2 - \alpha^t Q^* (P^*)^{-1} \alpha)\]
Again, we use $\xi_H$ is invariant under $M$, to get \eqref{alpha-inv-2}, taking conjugate we have
\[ \bar \alpha^t = \bar \alpha^t P^* - \alpha^t Q^* \]
Hence 
\[ |\alpha|^2 - \alpha^t Q^* (P^*)^{-1} \alpha = |\alpha|^2 - (\bar \alpha^t P^* -   \bar \alpha^t)  (P^*)^{-1} \alpha = \bar \alpha^t (P^*)^{-1} \alpha \]
Thus, doing the complex Gaussian integral, and note that $(-1)^{n/2}$ from determinant Hessian, should cancels with $i^n$ coming from the volume form, we get
\[ \la M^{-1} v, P_\Xi v \ra = \kk^{-m-1/2} (\bar \alpha (P^*)^{-1} \alpha)^{-1/2}  (\det P^*)^{-1/2}  \]
\epf



\begin{thebibliography}{HHHH}

\bibitem[AM78]{AM78}  R. Abraham and J. E.  Marsden, {\it Foundations of mechanics}. Benjamin/Cummings Publishing Co., Inc., Advanced Book Program, Reading, Mass., 1978.



 
 






\bibitem[BPU98]{BPU98}  D. Borthwick, T.  Paul, and A. Uribe,  Semiclassical spectral estimates for Toeplitz operators. Ann. Inst. Fourier (Grenoble) 48 (1998), no. 4, 1189-1229.



\bibitem[BG81]{BG81} L.  Boutet de Monvel and V.  Guillemin, {\it The SpectralTheory of Toeplitz Operators}, Ann.\ Math.\ Studies 99, Princeton Univ.\
Press, Princeton, 1981.

\bibitem[BSj]{BSj} L. Boutet de Monvel and J. Sj\"ostrand, \textit{Sur la
singularit\'e des noyaux de Bergman et de Szeg\"o}, Asterisque \textbf{34}--\textbf{35} (1976), 123--164.









\bibitem[Dau80]{Dau80}  I. Daubechies,  Coherent states and projective representation of the linear canonical transformations. J. Math. Phys. 21 (1980), no. 6, 1377-1389.



\bibitem[deG]{deG} M. de Gosson,  {\it  Symplectic geometry and quantum mechanics}. Operator Theory: Advances and Applications, 166. Advances in Partial Differential Equations (Basel). Birkhäuser Verlag, Basel, 2006



\bibitem[Del]{Del} H.  Delin,
\newblock Pointwise estimates for the weighted {B}ergman projection kernel in
  {$\mathbf C^n$}, using a weighted {$L^2$} estimate for the
  {$\overline\partial$} equation.
\newblock {\em Ann. Inst. Fourier (Grenoble)}, 48(4):967--997, 1998.



\bibitem[DSj]{DSj} M. Dimassi and J. Sjoestrand, {\it Spectral asymptotics in the semi-classical limit}. London Mathematical Society Lecture Note Series, 


\bibitem[F89]{F89} G. Folland, Harmonic Analysis in Phase Space, Ann. of Math. Stud., vol. 122, Princeton University Press, 1989.







\bibitem[HSj16]{HSj16} M. Hitrik and J. Sjoestrand, Two Minicourses on Analytic Microlocal Analysis, to appear
in {\it Algebraic and Analytic Microlocal Analysis}, M Hitrik, D. Tamarkin, B. Tysgan and S. Zelditch (eds.).


 \bibitem[HIII]{HIII} H\"ormander, L., {\em The analysis of linear partial differential operators.III. Pseudo-differential operators}. Classics in Mathematics. Springer, Berlin, 2007.









\bibitem[L02]{L02} Y. Long,  {\it Index theory for symplectic paths with applications}. Progress in Mathematics, 207. Birkhauser Verlag, Basel, 2002.


\bibitem[LuSh15]{LuSh15}  Z. Lu and B.  Shiffman,  Asymptotic expansion of the off-diagonal Bergman kernel on compact \Kahler manifolds. J. Geom. Anal. 25 (2015), no. 2, 761-782. 






\bibitem[P12]{P12} R.Paoletti,
Scaling asymptotics for quantized Hamiltonian flows. 
Internat. J. Math. 23 (2012), no. 10, 1250102.

\bibitem[P14]{P14} R. Paoletti, 
Local scaling asymptotics in phase space and time in Berezin-Toeplitz quantization. 
Internat. J. Math. 25 (2014), no. 6, 1450060, 40 pp


\bibitem[PP98]{PP98}  V. Petkov and G.  Popov, Semi-classical trace formula and clustering of eigenvalues for Schroedinger operators. Ann. Inst. H. Poincare Phys. Théor. 68 (1998), no. 1, 17-83.

\bibitem[PR85]{PR85} V.  Petkov and D.  Robert, Asymptotique semi-classique du spectre d'hamiltoniens quantiques et trajectoires classiques periodiques. Comm. Partial Differential Equations 10 (1985), no. 4, 365-390.

\bibitem[R87]{R87} D. Robert,{\it
Autour de l'approximation semi-classique. }
Progress in Mathematics, 68. Birkhäuser Boston, Inc., Boston, MA, 1987.

\bibitem[RZ]{RZ} Y. A. Rubinstein and S.  Zelditch, The Cauchy problem for the homogeneous Monge-Ampère equation, I. Toeplitz quantization. J. Differential Geom. 90 (2012), no. 2, 303-327.

\bibitem[Sa]{Sa} Safarov, Y. G.,  {\em Asymptotics of a spectral function of a positive elliptic operator
 without a nontrapping condition}, 
(Russian)  Funktsional. Anal. i Prilozhen.  {\bf 22}  (1988),  53--65, 96; translation in  Funct. Anal. Appl.  22  (1988),  no. 3, 213--223 (1989).
\bibitem[SV]{SV} Safarov, Y. G.;  Vassiliev, D.,  {\em The asymptotic distribution of eigenvalues of partial differential
 operators}, 
Translated from the Russian manuscript by the authors.
Translations of Mathematical Monographs, {\bf 155} American Mathematical Society, Providence, RI,  1997.


\bibitem[S93]{S93} E. Stein. Harmonic Analysis: Real-Variable Methods, Orthogonality, and Oscillatory Integrals.  Princeton University Press. 1993

\bibitem[ShZ02]{ShZ02} Shiffman, Bernard; Zelditch, Steve Asymptotics of almost holomorphic sections of ample line bundles on symplectic manifolds. J. Reine Angew. Math. 544 (2002), 181-222.



















\bibitem[Z97]{Z97} S. Zelditch, Index and dynamics of quantized contact transformations,
Ann. Inst. Fourier 47 (1997), 305--363, MR1437187, Zbl 0865.47018.


\bibitem[ZZ16]{ZZ16} S. Zelditch and P. Zhou, Interface asymptotics of partial
Bergman kernels on $S^1$-symmetric \Kahler manifolds, to appear in J. Symp. Geom.  (arXiv:1604.06655).

\bibitem[ZZ17]{ZZ17} S. Zelditch and P. Zhou, Central Limit theorem for spectral Partial  Bergman kernels (arXiv:1708.09267).



\end{thebibliography}
\end{document}